\documentclass[a4paper,11pt]{amsart}

\usepackage[utf8]{inputenc}
\usepackage[cyr]{aeguill}
\usepackage[english]{babel}
\usepackage{tikz}
\usepackage[all]{xy}

\usepackage[T1]{fontenc}
\usepackage{amsmath}
\usepackage{amsfonts}
\usepackage{amssymb}
\usepackage{amsthm,eepic}
\usepackage{mathrsfs}
\usepackage{enumitem}
\usepackage{graphicx}
\usepackage{footnote}
\usepackage{hyperref}
\hypersetup{colorlinks=true,    linkcolor=blue,
citecolor=red,       filecolor=BrickRed,       urlcolor=darkgreen
}
\usepackage[normalem]{ulem}
\usepackage{cancel}

\renewcommand{\geq}{\geqslant}
\renewcommand{\leq}{\leqslant}

\newcommand{\sigmaq}{\sigma_{q}}
\newcommand{\val}{\operatorname{val}}
\newcommand{\SL}{\operatorname{SL}}

\newcommand{\GL}{\operatorname{GL}}
\newcommand{\Gdiag}{G_{\mathcal{D}}}
\newcommand{\Hdiag}{H_{\mathcal{D}}}
\newcommand{\Gunip}{G_{\mathcal{U}}}
\newcommand{\Hunip}{H_{\mathcal{U}}}

\newtheorem{thm}{Theorem}[section]
\newtheorem{theo}[thm]{Theorem}
\newtheorem{defn}[thm]{Definition}
\newtheorem{rem}[thm]{Remark}

\newtheorem{prop}[thm]{Proposition}
\newtheorem{lem}[thm]{Lemma}

\newtheorem{ex}[thm]{Example}

\newtheorem*{thmintro}{Theorem}

\def\k{\mathbf{k}}
\def\Rb{\mathbf{R}}
\def\L{\mathbf{L}}
\def\GL{\mathrm{GL}}
\def\bn{\mathbf{s}}
\def\baralp{\bar{\alpha}}

\def\Z{\mathbb{Z}}
\def\N{\mathbb{N}}
\def\C{\mathbb{C}}
\def\R{\mathbb{R}}
\def\Q{\mathbb{Q}}
\def\P1{\mathbb{P}^{1}}

\makeatletter

\@addtoreset{equation}{section}
\makeatother

\makeatletter
\def\testb#1{\testb@i#1,,\@nil}%
\def\testb@i#1,#2,#3\@nil{%
  \draw[->, thick] (O) --++(#1);
  \ifx\relax#2\relax\else\testb@i#2,#3\@nil\fi}
\makeatother   
 
\title[On the computation of the difference Galois groups of order three eq.]{On the computation of the difference Galois groups of order three equations.}

\author{Thomas Dreyfus}
\address{Institut de Recherche Math\'ematique Avanc\'ee, UMR 7501 Universit\'e de Strasbourg et CNRS, 7 rue Ren\'e Descartes 67084 Strasbourg, France}
\email{dreyfus@math.unistra.fr}
\author{Marina Poulet }
\address{Aix-Marseille Université, France}
\email{marina.poulet@univ-amu.fr}

\thanks{ This project has received funding from the ANR De rerum natura project (ANR-19-CE40-0018). }
\date{\today}

\begin{document}

\begin{abstract} 
In this paper we consider the problem of computing the difference Galois groups of order three equations for a large class of difference operators including the shift operator (Case S), the $q$-difference operator (Case Q),  the Mahler operator (Case M) and the elliptic case (Case~E).  We show that the general problem can be reduced to several ancillary problems. We prove criteria to detect the irreducible and imprimitive Galois groups.  Finally,  we give a sufficient condition  of differential transcendence of solutions of order three difference equations.  We also compute the difference Galois group of an equation suggested by Wadim Zudilin. 
\end{abstract}

\maketitle

\section{Introduction}

Difference Galois theory is the subject of many recent papers because, for instance, of its application to (differential) transcendence of solutions of difference equations. More precisely, consider a field $\mathbf{k}$ equipped with an automorphism $\phi$. Let $C\subset \C$ be an algebraically closed field of characteristic $0$.  Typical examples we are going to consider in this paper are 
\begin{itemize}
\item (Case S, shift)  $\k=C(z)$ and $\phi Y(z)=Y(z+h)$  with $h\in C^{*}$;
\item  (Case Q, q-difference) $\k=C\left(z^{1/*}\right):=\displaystyle \cup_{\ell=1}^{\infty}C(z^{1/\ell})$   and 
$\phi Y(z)=Y(qz)$  with $q\in C^{*}$ that is not a root of unity;
\item (Case M, Mahler) $\k=C\left(z^{1/*}\right)$  and 
$\phi Y(z)=Y(z^{p})$  with $p\ge 2$ an integer;
\item (Case E, elliptic) Let $\Lambda$ be a lattice of $\C$. Without loss of generality, we assume that $\Lambda = \Z + \Z\tau$ with $\mbox{Im}(\tau)>0$. Let $\wp(z)$ be the corresponding Weierstrass function. In this example, $\k=\mathcal{M}er (\Lambda^{\infty}):=\displaystyle \cup_{\ell=1}^{\infty}C(\wp(z/\ell),\partial_{z}\wp(z/\ell))$ 
 and  $\phi Y(z)=Y(z+h)$  with $h\in \C^{*}$ such that  $h\Z\cap \Lambda= \{0\}$.
\end{itemize}
Then, consider the difference system 
$$\phi Y=AY, \quad A\in \GL_n(\mathbf{k}).$$
The algebraic relations between functions which are entries of matrix solutions of the difference system are reflected by a difference Galois group, which is an algebraic group attached to a difference equation (see for example \cite{van2006galois,Phi15}). Roughly speaking, if the difference Galois group is sufficiently big then there are no algebraic relations between these functions. Moreover, thanks to a differential Galois theory for difference equations (see \cite{HS08}), one can attach a differential algebraic group to a difference equation, which is Zariski-dense in the difference Galois group. Under  conditions on the difference Galois groups, this theory may also prove that some solutions of linear difference equations do not satisfy algebraic differential equations  (see \cite{HS08,arreche2017galois,dreyfus2018hypertranscendence,
adamczewski2020algebraic,adamczewski2021hypertranscendence,
arreche2021differential,dreyfus2021functional}) and has been recently used to prove that, in some cases, the generating series of walks in the quarter plane are differentially transcendental, see for instance \cite{dreyfus2018nature}.  There are many other applications of difference Galois theory that motivate the computation of the difference Galois groups. We could also mention the computation of the difference Galois groups of $q$-hypergeometric equations,  see \cite{roques2011generalized}  and density theorems \cite{ramis2006q, poulet2020density} but the present introduction does not pretend to be exhaustive.

Toward the proof of differential transcendence, we often need to prove that the difference Galois group is sufficiently  big, making it important to be able to compute the latter. However, the computation of difference Galois groups is in general a difficult task. In the literature one can find many algorithms for the computation of Galois groups. The most developed area is the one of differential equations: many algorithms are given to compute differential Galois groups (which are algebraic groups attached to a linear differential system), see for example \cite{kovacic1986algorithm,singer1993galois,CS99,BCWV,dreyfus2019computing}.
There is a general algorithm for  the computation of the Galois groups of linear differential equations of an arbitrary order but it is not effective  (see for example the algorithm of Hrushovski in \cite{Hru2002} and some improvements  \cite{Feng15,amzallag2022degree}).  In contrast, the case of difference equations is less developed.  
Algorithms have been given for the computation of difference equations of order $1$ and $2$ in these four famous cases: the shift (see \cite{van2006galois,hendriks1998algorithm}), the $q$-difference (see \cite{hendriks1997algorithm}), the Mahler (see \cite{roques2018algebraic}) and the elliptic cases (see \cite{dreyfus2015galois}). For the shift case, an analogue of Hrushovski's algorithm has been developed in \cite{Feng18}. It computes the Galois groups for equations of arbitrary order whose coefficients are rational functions but is still inefficient. However, in general, for the Galois groups of difference equations of order  $3$ or more, only few things are known.  The aim of this article is to explain how one can compute Galois groups of difference equations of order $3$ under mild assumptions on the difference field to which the coefficients of the difference equations belong. In particular, it takes into account the four previous difference equations. We highlight that we are mainly concerned by the theoretical aspects of the algorithm and no study of complexity will be made.   Throughout this paper, we explain what we need to be able to do for using it in practice.   In fact, the key point to apply it concretely is to be able to find rational (resp. hyperexponential)  solutions to certain difference equations.

The starting point of this paper was  an example suggested by Wadim Zudilin. This example is  running as a common thread throughout this paper. The following double sum is related to the Rogers-Ramanujan identities 
$$U(z,t)=\displaystyle \sum_{m,n=0}^{\infty}\frac{q^{m^2+3mn+3n^2} z^{m}t^{n}}{(q;q)_m (q^3;q^3)_n}, \quad \hbox{ with } (\star;q)_n=(1-\star)\times \dots \times (1-\star q^{n-1}).$$
For each fixed value of $t$, $z\mapsto U(z,t)$ satisfies a $q$-difference equation:
\begin{equation}\label{eq0}
ty(q^3 z)  -(t+qt+q^4 z^2)y(q^2 z) +q(t-q^2 z)y(qz) +q^3 z y(z)=0.
\end{equation}
One of the questions asked by Zudilin was the computation of the Galois groups for the values $t=1$ and $t=q^2$. We are going to compute the difference Galois groups over $\C(z^{1/*})$ for every value of $t\neq 0$ and see that it does not depend on the parameter $t$.  Note that when $t=0$, \eqref{eq0} is a $q$-difference equation of order $2$ and the computation could be made using \cite{hendriks1998algorithm}. 
\begin{thmintro}[Theorem \ref{thm2}]
For all $t\in \C^*$, the difference Galois group of \eqref{eq0} is $\GL_3 (\C)$.  
\end{thmintro}

Using \cite[Corollary 3.1]{dreyfus2021functional} and Theorem \ref{thm2}, we may  prove that for all $t\in \C^*$,  there are no nontrivial algebraic differential relations between $U(z,t)$,  $U(qz,t)$ and $U(q^2 z,t)$ over $\C(z^{1/*})$.

Now, we describe more precisely the content of this paper.  After some reminders about the Galois theory of difference equations in Section~\ref{sec:GaloisGp}, we recall a criterion in Section~\ref{secfact} to know if the Galois group $G$ of a difference equation is reducible or not. In Section~\ref{sec:theoryPractice}, we explain what we need, as a framework, in order to compute $G$ in practice. In Section~\ref{sec:GaloisGpDiffceSystems} we prove general results about Galois groups of difference systems of arbitrary order. We explain the strategy for computing $G$ if $G$ is reducible in Section \ref{sec:reducible case}. If $G$ is not reducible we distinguish two cases: the imprimitive case and the primitive case. For the imprimitive case, we give information (the connected component of the identity and its index in $G$) on the possible  Galois groups which occur when the order of the difference system is a prime number in Section~\ref{sec:imprimitive case}. Moreover, we know which information correspond to $G$ thanks to the knowledge of the determinant of $G$ (which is the Galois group of an equation of order $1$). 
After these general results, we recall the strategies for computing the Galois groups of the difference equations of order $1$ and the diagonal systems in Section~\ref{sec:GaloisGpsOrderOneEtDiag} and the strategies for the difference equations of order $2$ in Section~\ref{secord2}. Section \ref{secord3} is devoted to the order three case. We give an application to the differential transcendence in Section~\ref{sec:appliDiffalTranscendence}.

\subsection*{Notations and conventions} In what follows,  all rings are commutative with identity and contain the field of rational numbers.  In particular, all fields are of characteristic zero. 

If $G$ is an algebraic group, we denote by $G^0$ the connected component of the identity of $G$.
Let $(\k,\phi)$ be a difference field, that is, a field equipped with an automorphism  and let $[A]_{\phi}$ be the difference system $\phi(Y)=AY$ where $A\in\GL_n\left(\k\right)$. To lighten notations, we denote it $[A]$ when there is no ambiguity.

\section{Difference algebra}
\label{sec:diffceAlgebra}
\subsection{Galois groups}
\label{sec:GaloisGp}

In this section, we make a short overview of the Galois theory of linear difference equations.  
For more details on what follows, we refer to \cite[Chapter~1]{van2006galois}.
A difference ring is a commutative ring $\Rb$ equipped with an automorphism $\phi : \Rb\rightarrow \Rb$. We denote by $C_{\Rb}:= \{c\in \Rb| \phi(c)=c\}$, the constants of the difference ring $(\Rb,\phi)$. A difference ideal $I$ of a difference ring $(\Rb,\phi)$ is an ideal such that $\phi(I)\subset I$.  If $\Rb$ is a field, we call $(\Rb,\phi)$ a difference field. In this case, $C_{\Rb}$ is a field, the field of constants. 
With an abuse of notations, we will often denote by $\Rb$ the difference ring $(\Rb,\phi)$. 

\medskip

\emph{Throughout this paper, we assume that $(\k,\phi)$ is a difference field such that its field of constants $C_{\k}$ is algebraically closed.}

\medskip

If we consider a difference system 
\begin{equation}
\label{eq:diffcesystembase}
    \phi(Y) = AY,\quad with \quad A\in \GL_n\left(\k\right),
\end{equation}
according to \cite[Section 1.1]{van2006galois}, there exists a difference ring extension $(\L,\phi)$ of $(\k,\phi)$ such that:
\begin{itemize}[label=$\bullet$]
    \item there exists a fundamental matrix of solutions of \eqref{eq:diffcesystembase} with entries in $\L$, that is, a matrix $U\in\GL_n\left(\L\right)$ such that $\phi\left(U\right)=AU$;
    \item $\L$ is generated as a $\k$-algebra by the entries of $U$ and $\det\left(U\right)^{-1}$;
    \item the only difference ideals of $\L$ are $\lbrace 0\rbrace$ and $\L$.
\end{itemize}
A difference ring $\L$ which satisfies these conditions is called a Picard-Vessiot ring for the system \eqref{eq:diffcesystembase} over $\k$. It is unique up to isomorphisms of difference rings over $\k$. We have the following property: $C_{\L} = C_{\k}$.

The difference Galois group $G$ of \eqref{eq:diffcesystembase} over $\k$ is the group of $\k$-linear automorphisms of $\L$ commuting with $\phi$, that is, 
$$
G:=\lbrace \sigma\in\text{Aut}(\L/\k)\mid \sigma\circ\phi=\phi\circ\sigma \rbrace \, .
$$
For any fundamental matrix $U \in \GL_n(\L)$, an easy computation shows that 
$U^{-1}\sigma(U) \in \GL_{n}(C_{\L})= \GL_{n}(C_{\k})$ for all $\sigma \in G$. 
By  \cite[Theorem~1.13]{van2006galois}, the faithful representation
$$
\begin{array}{rccl}
 \rho: & G & \rightarrow & \GL_{n}(C_{\k}) \\ 
   & \sigma & \mapsto & U^{-1}\sigma(U)
\end{array}
$$
identifies $G $ with a linear algebraic subgroup $G\subset\GL_{n}(C_{\k})$.  If we take another fundamental matrix of solutions $V$, we find an algebraic group that is conjugated to the first one. 
In this paper, by ``computation of the difference Galois group $G$'', we mean ``computation of an algebraic subgroup of $\GL_n\left(C_{\k}\right)$ conjugated to $\rho(G)$'', thus we work up to conjugations.  Abusing notations, we will still denote by $G$ the image of the difference Galois group under the above representation.\par

We denote by $\k\left< \phi,\phi^{-1}\right> $ the \"Ore algebra of noncommutative Laurent polynomials with coefficients in $\k$ such that $\phi f = \phi(f)\phi$ for all $f\in\k$.
We consider the difference equation 
\begin{equation}\label{eq:diffceeqbase}
Ly=0,\quad  \text{where}\quad  L:=\sum_{i=0}^{n}  a_i \phi^{i} \in \k\left< \phi\right> \text{ with } a_0 a_n \neq 0 \, ,
\end{equation}
that is, 
$a_0 y+\dots +a_n \phi^n (y)=0$ with $a_i\in \k$ and $a_0 a_n \neq 0 $.
We transform the difference equation \eqref{eq:diffceeqbase}
into the difference system 
$$\phi(Y)=A_L Y,  \quad\hbox{ where} \quad A_L=\begin{pmatrix}
&1&&\\ 
 &  &\ddots & \\
&&&1\\
-a_0/a_n& \dots&&-a_{n-1}/a_n
\end{pmatrix} $$
is the companion matrix associated with $L$.
By definition, the difference Galois group of \eqref{eq:diffceeqbase} is the difference Galois group of the system $[A_L]$.\par 

The change of variables $Z:=TY$,  $T\in \GL_{n}(\k)$, transforms the system  $\phi Y=AY$ with  $A\in  \GL_{n}(\k)$ into $\phi Z=BZ$ where $B:=\phi(T)AT^{-1}\in  \GL_{n}(\k)$.  Let  $A,B\in  \GL_{n}(\k)$.  We will say that $[A]$ and $[B]$ are equivalent over $\k$ if there exists   $T\in \GL_{n}(\k)$ such that $B=\phi(T)AT^{-1}$. \par 

In what follows, we assume moreover that the field $\k$ satisfies:
\begin{itemize}
    \item $\k$ is a $\mathcal{C}^{1}$-field, that is every non-constant homogeneous polynomial $P$ over $\k$ has a nontrivial zero provided that the number of its variables is more than its degree;
    \item $\k$ is the only finite algebraic difference field extension of $\k$.
\end{itemize}
All the above mentioned examples (S,Q,M,E) satisfy these two conditions, see \cite{hendriks1997algorithm,hendriks1998algorithm,dreyfus2015galois,roques2018algebraic}. 
Note that in the $q$-difference case, the field $(C(z),\phi)$ is a difference field but there are nontrivial finite  algebraic difference field extensions. For instance $C(z^{1/2})|C(z)$ is of degree two and we may extend $\phi$ on it with $\phi (z^{1/2})=q^{1/2}z^{1/2}$. This is the reason why we consider $C(z^{1/*})$ instead of $C(z)$ in that setting. Similar considerations occur in the elliptic case.

We recall that $C:= C_{\k}$ is by hypothesis algebraically closed and the characteristic of $\k$ is zero. 
The following statement will be crucial in what follows.  See \cite[Propositions~1.20 and 1.21]{van2006galois}, see also  \cite[Theorem 2.4]{dreyfus2015galois}.  
\begin{theo}
\label{theo:reduced form}
Let $G \subset \GL_{n}(C)$ be the difference Galois group over $(\k,\phi)$ of \eqref{eq:diffcesystembase}. 
Then, the following properties hold:
\begin{itemize}
\item $G/G^0$ is finite and cyclic, where $G^0$ is the identity component of $G$;
\item there exists $B \in G(\k)$ such that \eqref{eq:diffcesystembase} is equivalent to $\phi Y=BY$ over~$\k$. 
\end{itemize}
Let $\widetilde{G}$ be an algebraic subgroup of $\GL_{n}(C)$ such that~$A\in \widetilde{G}(\k)$. The following properties hold:
\begin{itemize}
\item $G$ is conjugate to a subgroup of $\widetilde{G}$;
\item any minimal element in the set of algebraic subgroups $\widetilde{H}$ of $\widetilde{G}$ for which there exists 
$T\in \GL_{n}(\k)$ such that $\phi(T)AT^{-1}\in \widetilde{H}(\k)$ is conjugate to~$G$;
\item $G$ is conjugate to $\widetilde{G}$ if and only if, for any $T\in  \widetilde{G}(\k)$ and for any proper algebraic subgroup $\widetilde{H}$ of $\widetilde{G}$, 
one has that $\phi(T)AT^{-1}\notin \widetilde{H}(\k)$.
\end{itemize}
\end{theo}

In practice, to compute the difference Galois group, we are going to compute a gauge transformation such that the transformed system lies in $G(\k)$.
This leads to the following definition.  

\begin{defn}
Let $G \subset \GL_{n}(C)$ be the difference Galois group over $(\k,\phi)$ of the system \eqref{eq:diffcesystembase}. We say that the system \eqref{eq:diffcesystembase} is \emph{in reduced form} (or \emph{reduced}) when $A\in G(\k)$.  
\end{defn} 

Given a difference system $\phi(Y)=AY$ and an integer $\ell\geq 1$,  we can iterate it: \begin{equation}
    \label{eq:iterateSystem}
    \phi^{\ell}(Y)=A_{[\ell]}Y, \quad  \text{where} \quad  A_{[\ell]}=\phi^{\ell-1}(A)\times \dots \times A\, .
\end{equation}
Let $G_{[\ell]}$ be the difference Galois group of the system \eqref{eq:iterateSystem}. 
The following lemma may be deduced from the proof \cite[Proposition~4.6]{dreyfus2021functional} in the particular case where the parametric operator is the identity,  and is proved in  \cite[Proposition~4.10]{adamczewski2021hypertranscendence}.
It compares the Galois groups $G_{[\ell]}$,  $\ell \geq 1$. 

\begin{lem}\label{lem2}
For all $\ell \geq 1$,  $G_{[\ell]}$ has the same connected component of the identity as $G$.  Furthermore,  there exists $\ell_0\geq 1$, such that $G_{[\ell_0]}$ is the connected component of the identity of $G$.
\end{lem}

\subsection{Factorisations of difference operators}\label{secfact}
In this subsection, we focus on the reducible difference operators and make the relation with the reducible Galois groups and the Riccati equation.  
Let $G$ be an algebraic subgroup of $\GL_{n}(C)$. We say that $G$ is reducible if there exists a non trivial vector subspace $V$ of $C^n$ such that for all $\sigma\in G$, $\sigma(V)\subset V$.  We say that $G$ is irreducible otherwise. \par 
The following lemma relates the irreducibility of the difference Galois group with the factorisation of the difference operator, see \cite[Lemma 4.4]{adamczewski2021hypertranscendence}.

\begin{lem}\label{lem1}
Let $G\subset \GL_{n}(C)$ be the difference Galois group of \eqref{eq:diffcesystembase}. Then,  $G$ is reducible with a nontrivial invariant space of dimension $0<m<n$, if and only if there exists $T\in \GL_n(\k)$,  $A_1\in \GL_m(\k)$, and $A_2,A_3$ matrices of convenient size,    such that 
$$\phi(T)AT^{-1}=\left(\begin{array}{cc}
A_1 &  A_2\\
0 & A_3  
\end{array}\right).
$$
If we consider the difference system $[A_L]$ attached to \eqref{eq:diffceeqbase}, then its difference Galois group is reducible with $A_1\in \GL_m(\k)$ if and only if
$L$ admits a nontrivial right factor of order $m$.
\end{lem}

Thus, to determine whether $G$ is irreducible or not, we have to determine whether the difference operator $L$ defined in \eqref{eq:diffceeqbase} admits a nontrivial right factor or not.  For right factors of order one, we have the following criterion. 

\begin{lem}[Lemma 2 in \cite{abramov2021rational}]\label{lem:eq left factor}
Let $L$ be the difference operator defined in \eqref{eq:diffceeqbase}.
The operator of order one $\phi-\alpha $ is a right factor of $L$ if and only if 
$$
\sum\limits_{i=0}^n a_i \prod\limits_{j=0}^{i-1} \phi^j(\alpha) = 0.
$$
\end{lem}

Although the above equation,  called the Riccati equation,   is nonlinear it might be solved in some situations we are going to describe later.\par 
When $L$ has order two,  it admits a nontrivial factor if and only if it admits a nontrivial right factor of order one. \par
When $L$ has order three, it admits a nontrivial factor if and only if it admits a nontrivial factor of order one.  The latter may be either a right or a left factor.  Let us see how we can pass from left factor to right factor.   
 We define the dual $M^{\vee} $ of an operator $M=\sum_{i=m}^{n}   a_{i}\phi^{i} \in\k\left< \phi,\phi^{-1}\right>$ by 
$$
M^{\vee} := \sum_{i=m}^{n}  \phi^{-i} (a_{i})
\phi^{-i}.$$ 
By linearity, if $M_1, M_2\in \k\left< \phi,\phi^{-1}\right>$, we can check that
$$
 (M_1 M_2)^{\vee}=M_2^{\vee}M_1^{\vee}. 
$$
Hence, an operator has a left factor of order one if and only if its dual has a right factor of order one.   So we will focus on right factors of order one. 

\begin{rem}
\label{rem:rightFactorOrder1}
Assume $n>1$ and that $L$ admits a right factor of order one so let us write $L$ as 
$\widetilde{L}(\phi-\alpha)$ where $\alpha\in \k$ and  $\widetilde{L}\in\k\left< \phi,\phi^{-1}\right>$ has order $n-1$.  
Let $y_1$ be a solution of $Ly=0$ such that $\phi(y_1)\neq \alpha y_1$ and let $y_2:=\phi(y_1)-\alpha y_1$. By construction $y_2\neq 0$ is a solution of $\widetilde{L}y=0$. If $Y:=\left(y_1,y_2,\ldots,\phi^{n-2}(y_2)\right)^t$ then $\phi(Y)=\widetilde{A}Y$ where
$$
\widetilde{A}:=\left(\begin{array}{c|cccc}
\alpha &  1 & 0 & \cdots & 0 \\
\hline
0 & & & &  \\
\vdots & & & A_{\widetilde{L}}&  \\
0 & & & & 
\end{array}\right)\, .
$$
Thus, the companion system associated with $Ly=0$ is equivalent to the system $[\widetilde{A}]$.  Furthermore, $A_{\widetilde{L}}$ is a companion matrix associated with $\widetilde{L}$.\end{rem}

\subsection{From theory to practice}
\label{sec:theoryPractice}
We recall that the aim of this paper is to explain how one can compute theoretically the Galois group of a difference equation of order $3$.
In practice we may also need to make more assumptions on the difference field $(\k,\phi)$ in order to compute effectively the difference Galois groups:
\begin{itemize}
\item
Assume that $\mathbf{k}$ is an effective field,
i.e. that one can compute representatives of the four operations $+,-,\times,/$ and one can effectively test whether two elements of $\mathbf{k}$ are equal.  
\item 
Assume that, given a homogeneous linear difference system $[A]$ with $A\in \GL_n(\mathbf{k})$, we can effectively find a 
basis of its rational solutions, i.e. its solutions $Y\in \mathbf{k}^n$.
\item Assume that given a difference operator $L$, if $L$ admits at least one right factor of order one, then we are able to compute effectively one of such right factors.
\end{itemize}

When $C=\overline{\Q}$,  for the cases (S, Q, M), the  two first points are satisfied. One can find some references for the second point in this table:
\medskip
\begin{center}
\begin{tabular}{|c|c|c|c|c|}
\hline
    cases &  (S) & (Q) & (M)&(E) \\
\hline 
    ref. & \cite{Abr95,AB98,vanHoeij1998,AB2001} & \cite{Abr95,AB2001,Abr02} & \cite{CDDM18} & \cite{dreyfus2015galois,combot2022hyperexponential}\\
\hline
\end{tabular}
\end{center}
\medskip 

Note that in the elliptic case the second  point is not satisfied, unless considering the case $n=2$, see \cite{dreyfus2015galois} (note that the latter is concerned by the more general problem of solving Riccati equations), or when the framework is modified,  see for instance \cite{combot2022hyperexponential}.  Thus, in that case, our algorithm may not lead to a systematic computation of the difference Galois group. It will reduce the computation to the determination of rational solutions of certain linear difference equations that may be complicated to solve concretely.
Throughout this paper, we will emphasize which kind of equations we have to be able to solve.

We saw in Section \ref{secfact} that the study of the irreducibility of the Galois groups is closely related to the factorisation of the difference operator.  From Lemma \ref{lem:eq left factor}, when we look at a right factor of order one, we have to solve a Riccati equation of the form  $$
\sum\limits_{i=0}^n a_i \prod\limits_{j=0}^{i-1} \phi^j(\alpha) = 0.
$$

When $C=\overline{\Q}$,  we are able to compute right-factors of order $1$ in the cases (S,Q,E) but we want to emphasis that the complexity is exponential, in the order of the equation, contrary to the search of rational solutions, whose complexity is polynomial. 
   Here-below, one can find some references for these computations:

\medskip
\begin{center}
\begin{tabular}{|c|c|c|c|}
\hline
    cases &  (S) & (Q) & (E)\\
\hline 
ref.  & \cite{Petkovsek92,CvH06,abramov2021rational} & \cite{APP98,CvH06,abramov2021rational} &\cite{dreyfus2015galois,combot2022hyperexponential} \\
\hline
\end{tabular}
\end{center}
\medskip 
For the Mahler case, the computation of Riccati solution is a work in progress by the first author of this paper, Chyzak, Dumas and Mezzarobba.
\medskip
\begin{ex}\label{ex1}
Let us now present the strategy of \cite{abramov2021rational} to show that the difference Galois group of  \eqref{eq0} over $\C(z^{1/*})$ is irreducible, which is equivalent to show that the corresponding operator has no factor of order $1$. We point out the fact that $\C(z^{1/*})$ is not an effective field but the strategy still work in that case. 
Assume that $t\neq 0$. 
 Let us begin by proving that there is no right factor of order one. Let 
 $$L=a_0+a_1 \phi+ a_2\phi^2+a_3 \phi^3$$
  be the $q$-difference operator corresponding to \eqref{eq0}, and let $L^{\vee} $ be its dual. 
By  \cite[Lemma 3]{abramov2021rational},
 For any $\alpha \in \k^{\times}$  there exist $\lambda \in \C^{\times}$, $\ell\in\N^*$ and unitary $b,c,r\in \C[z^{1/\ell}] \setminus \{0\}$  such that 
$$
  \alpha =\lambda\hspace{0.1cm}  \frac{b}{c} \hspace{0.1cm}\frac{\phi(r)}{r}
$$
 and 
\begin{itemize}
\item $\gcd (b, \phi^k (c))=1$ for all $k\in \lbrace 0,1, 2\rbrace$;
\item $\gcd (b, r)=1$;
\item $\gcd (c, \phi (r))=1$.
\end{itemize}
Using this, one can deduce by  \cite[Lemma 3]{abramov2021rational} that if $\phi-\alpha$ is a right factor of order $1$ of $L$ for some  $\alpha \in \k^{\times}$, then, in the above decomposition for $\alpha$, we have $b \vert a_{0}$,   $c \vert \phi^{-2}(a_{3})$,  and 
\begin{equation}\label{eq:bla}
 \lambda^3  a_{3} b \phi(b)\phi^2 (b)  \phi^3 (r) +\lambda^2 a_{2}  b \phi(b)\phi^2 (c) \phi^2 (r)+\lambda a_{1} b \phi(c)\phi^2 (c)  \phi (r)+ a_{0} c \phi(c)\phi^2 (c) r=0.
\end{equation}

The strategy to decide whether $L$ has a right factor of order one is the following.  For every unitary pair   $(b,c)\in \left( \C[z^{1/\ell}]\setminus \{0\}\right)^2$ such that $b \vert a_{0}$,   $c \vert \phi^{-2}(a_{3})$,  we decide whether there exist $\lambda \in \C^{\times}$ and $r\in \C[z^{1/\ell}] \setminus \{0\}$ such that \eqref{eq:bla} holds. If we find a solution then we have a right factor of order one, otherwise there is no right factor of order one.
Here, since $a_0=q^3 z$, $a_3=t$ are monomials in $z$ (recall that $t\neq 0)$ the choices are very limited. The only possibilities for the pair $(b,c)$ are $(z^d,1)$, with $0\leq d\leq 1 $.   Recall that $\deg(a_3)=0$,  $\deg(a_2)=2$,  $\deg(a_1)=\deg(a_0)=1$.  
The degrees of the terms in \eqref{eq:bla} are respectively given by  
$$3d+\deg r(z), \quad 2+2d+\deg r(z), \quad 1+d+\deg r(z), \quad 1+\deg r(z)$$
and a direct calculation shows that, for $0\leq d\leq 1 $,  $2+2d+\deg r(z)$ is greater than the remaining ones, whence there is no solution.
  We proceed similarly with the dual for the left factor of length one.
  \begin{multline*}
 L^{\vee}=a_{3}(q^{-3}z)\sigmaq^{-3}  + a_{2}(q^{-2}z)\sigmaq^{-2} +a_{1}(q^{-1}z)\sigmaq^{-1} +a_{0}(z)\\
=\sigmaq^{-3}\Big(a_{3}(z)  +a_{2}(qz)\sigmaq +a_{1}(q^{2}z)\sigmaq^{2} +a_{0}(q^{3}z)\sigmaq^{3}\Big)
\end{multline*} 
Here the possible pairs for $(b,c)$ are $( 1,z^d)$, with $0\leq d\leq 1 $. The degrees of the terms of the equation corresponding to \eqref{eq:bla}  are
$$1+\deg r(z), \quad 1+d+\deg r(z), \quad 2+2d+\deg r(z), \quad 3d+ \deg r(z) .$$
Again, for $0\leq d\leq 1 $,  $2+2d+\deg r(z)$ is greater than the remaining ones.
\end{ex}

\section{Galois groups of difference systems}\label{sec:GaloisGpDiffceSystems}

Consider the difference system \eqref{eq:diffcesystembase} and let 
 $G\subset \GL_{n}(C)$ be its Galois group.   In order to compute $G$, we follow the strategy of the seminal paper \cite{hendriks1998algorithm}, that is,  computing a gauge transformation $T\in \mathrm{GL}_n (\k)$ such that $[\phi(T)AT^{-1}]$ is in reduced form.  We will study the following cases (definitions of reducible and imprimitive groups are given in the corresponding section): 
 \begin{itemize}
 \item $G$ is reducible;
 \item $G$ is irreducible and imprimitive;
 \item $G$ is irreducible and primitive. 
 \end{itemize}
Before considering the particular case of difference equations of order $2$ and $3$, the aim of this section is to prove some general results for the Galois groups $G$ of systems of the form \eqref{eq:diffcesystembase}, with $n\geq 2$.  

\subsection{The reducible case}
\label{sec:reducible case}

Let us treat the case where the Galois group $G$ is reducible.  By Lemma  \ref{lem1},  $G$ is reducible if and only if there exist $T\in \mathrm{GL}_n(\k)$,  $A_1\in \mathrm{GL}_m(\k)$, $m<n$,  and $A_2,A_3$ matrices of convenient size,    such that 
$$\phi(T)AT^{-1}=\left(\begin{array}{cc}
A_1 &  A_2\\
0 & A_3  
\end{array}\right).
$$

The strategy of reduction will follow the one given in the differential setting in \cite[Section 5]{dreyfus2019computing}.   Let us assume that the bloc diagonal system $[D]$, where $D:=\begin{pmatrix}
A_1&0 \\ 
0 &A_3 
\end{pmatrix}$, is in reduced form and let us see how to reduce $[A]$.  In general, $[D]$ is not in reduced form and we have to reduce the bloc diagonal system with a bloc diagonal gauge transformation. 
For small orders, we will see how to do this in practice but we want to emphasize the fact that reducing the bloc diagonal system is a hard task in general.   We denote by $\Gdiag$ the Galois group of $[D]$.

\begin{lem}
 \label{lem:Gauge transfo reduc}
 Let $B$ be a matrix of the form $\begin{pmatrix}
 A_1 & A_2 \\ 0 & A_3
 \end{pmatrix}$ where $D:= \begin{pmatrix}
 A_1 & 0 \\ 0 & A_3
 \end{pmatrix}$ is such that the system $[D]$ is in reduced form. 
A matrix $T\in \GL_n(\k)$ such that $\left[\phi(T)BT^{-1}\right]$ is reduced can be found of the form $T=\begin{pmatrix}
\mathrm{Id}_m & \star \\ 
0 & \mathrm{Id}_{n-m}\\
\end{pmatrix}$. 
\end{lem}

\begin{proof}
It is the same kind of results as  \cite[Theorem 2.4]{dreyfus2019computing} adapted to this difference context.  We denote by $G$ (respectively by $\Gdiag$) the Galois group of $[B]$  (respectively of $[D]$). Let $H$ be the smallest algebraic group of upper bloc triangular matrices such that $B\in H(\k)$. It is an algebraic group whose elements are of the form $\begin{pmatrix} \mathscr{H}_1 & \mathscr{H}_2 \\ 0 & \mathscr{H}_3 \end{pmatrix}$. We denote by $\Hdiag$ the group whose matrices are of the form $\begin{pmatrix} \mathscr{H}_1 & 0\\ 0 & \mathscr{H}_3 \end{pmatrix} $ where $\begin{pmatrix} \mathscr{H}_1 & \mathscr{H}_2\\ 0 & \mathscr{H}_3 \end{pmatrix}\in H $ for a matrix~$\mathscr{H}_2$, it is a block diagonal algebraic group.
Let $\Hunip$ (respectively $\Gunip$) be the subgroup of $H$ (respectively $G$) whose elements are of the form $\begin{pmatrix}
\mathrm{Id}_m & \star \\ 0 & \mathrm{Id}_{n-m}
\end{pmatrix}$. We have $\Hdiag\sim H/\Hunip$. 
Let us prove that the group $\Hdiag$ is the smallest algebraic group 
such that $D\in \Hdiag (\k)$. Indeed, considering the projection morphism $ \pi : H \rightarrow  \Hdiag$, where we identify $H/\Hunip$ and $\Hdiag$, if there exists an algebraic group $\widetilde{\Hdiag}\subsetneq \Hdiag$ such that $D\in \widetilde{\Hdiag}(\k)$, then the algebraic group  $\widetilde{H}:=\pi^{-1}(\widetilde{\Hdiag})\subsetneq H$ is such that $B\in \widetilde{H}(\k)$, which contradicts the minimality of $H$.
  Since $[D]$ is in reduced form,   we have  $\Gdiag= \Hdiag$.  
Let $R = \begin{pmatrix} R_1 & R_2 \\ 0 & R_3 \end{pmatrix}\in H(\k)$ be such that the system $\left[\phi(R)BR^{-1}\right]$ is in reduced form. Let $R_{\mathcal{D}} :=\begin{pmatrix} R_1 & 0 \\ 0 & R_3 \end{pmatrix} \in \Hdiag(\k)=\Gdiag(\k)$.   

Since $\Gdiag\sim G/\Gunip$, there exists $Q_2$, such that $Q:= \begin{pmatrix}
R_1 & Q_2 \\ 0 & R_3
\end{pmatrix}\in G(\k)$. Therefore $Q^{-1} = \begin{pmatrix}
R_1^{-1} &\star \\ 0 & R_3^{-1}
\end{pmatrix}\in G(\k)$. A gauge transformation of $\phi(R)BR^{-1}\in G(\k)$ by the matrix  $Q^{-1}\in G(\k)$ gives an element of $G(\k)$. Hence $\left[\phi(T)BT^{-1}\right]$ is in reduced form where $T:= Q^{-1}R = \begin{pmatrix}
\mathrm{Id}_m & \star \\ 0 & \mathrm{Id}_{n-m}
\end{pmatrix}$.
\end{proof}

\begin{rem}\label{rem1}
As in \cite[Section 5]{dreyfus2019computing},  we may also consider the case where the system has more than two blocs.   More precisely, consider a bloc triangular system $[A]$, where  $$A=\left(\begin{array}{cccc}
A_1 &  A_{1,2}&\dots  &\\
 & A_2 &   A_{2,3}\\
&&\ddots& \\
0&&&A_k
\end{array}\right).$$
Let us assume that the bloc diagonal system $[D]$, with $D=\mathrm{Diag}(A_1,\dots, A_k)$, is in reduced form.   Then, the proof of Lemma \ref{lem:Gauge transfo reduc} may be adapted in that context and we find a gauge transformation $T=\left(\begin{array}{cccc}
\mathrm{Id} &  \star &\dots  & \star\\
0 & \mathrm{Id} & \ddots & \vdots \\
&&\ddots&  \star  \\
&&&\mathrm{Id}
\end{array}\right)\in \GL_n (\mathbf{k})$ such that $[\phi(T)A T^{-1}]$ is in reduced form.  
\end{rem}

\subsection{The imprimitive case}
\label{sec:imprimitive case}

 Let $G$ be an algebraic subgroup of $\GL_{n}(C)$.  We say that $G$ is imprimitive if there exists a nontrivial decomposition into $C$-vector spaces  $C^n =V_1 \oplus \dots \oplus V_k$, such that any element of $G$ induces a permutation between the $V_i$.  We say that $G$ is primitive otherwise.  

Consider $G$, an imprimitive group and let us write a decomposition   $C^n =V_1 \oplus \dots \oplus V_k$.  The image of $V_1$ under  an element of $G$ is of the form $V_j$,  where $V_j$ and $V_1$ have the same dimension. If we further assume that $G$ is irreducible, for all $1\leq j \leq k$, there exists $\sigma\in G$, such that  $\sigma (V_1)=V_j$, proving that each $V_j$ has the same dimension and thus is a divisor of the order $n$ of the equation \eqref{eq:diffcesystembase}.   Therefore,  when $n$ is prime and $G$ is irreducible,  the situation is simplified since every space has dimension one. 
In what follows, we  give information on the irreducible imprimitive groups that may occur as difference Galois groups when the order $n$ of the system \eqref{eq:diffcesystembase} is prime (see Theorem \ref{thm1}). Moreover, $\det(G)$ indicates which one of this information correspond to the Galois group $G$  (see Proposition~\ref{prop:detDistinguish}). More precisely, Theorem \ref{thm1} and Proposition~\ref{prop:detDistinguish} imply the following result.

\begin{thm}\label{thmimpr}
Assume that $n$ is a prime number.  Let $G \subset \mathrm{GL}_n (C)$ be an imprimitive irreducible group such that $G/G^{0}$ is finite and cyclic. 
\begin{enumerate}
 \item If $\det(G) = C^*$ then $G^0\sim (C^*)^n$ and  
$G/G^0 \sim  \Z/ n \Z$.
 \item If $\det(G)=\Z/  \nu \Z$ then ${G^0\sim (C^*)^{n-1}}$ and
$G/G^0 \sim  \Z/ ns \Z$  where
 $$s = \left\lbrace \begin{array}{ll}
      \nu/n & \text{ if }  \nu \in n\N^* \\
     \nu & \text{ otherwise. }
 \end{array}\right.$$
\end{enumerate}
\end{thm}

At the end of this section, we explain how to compute a reduced form and a gauge transformation which gives this reduced form (see Proposition \ref{prop2}). \par 

\subsubsection{Imprimitive groups that may occur}
 
 Let us now describe the possible irreducible imprimitive groups that may occur as difference Galois groups when $n$ is prime.  Let us begin by an elementary lemma that will be used several times. Note that the assumption $n$ prime is not required for this lemma.
 
\begin{lem}
\label{lem:irred_primitive_G}
Let $G \subset \mathrm{GL}_n (C)$ be an irreducible group and let $H\triangleleft G$ be such that $G/H$ is finite and cyclic. Then, $H$ is not composed only of dilatations.
\end{lem}

\begin{proof}
By contradiction, assume that $H$ is composed only of dilatations. Since $G/H$ is finite and cyclic,  there exists a matrix $\mathscr{F}$ such that $G=H\cup H \mathscr{F} \cup\dots  \cup H \mathscr{F}^{m-1}$ for some positive integer $m$.  The matrix $\mathscr{F}$ is triangularisable and since $H$ is formed by dilatations,  it implies that there is a common basis of triangularisation of every elements of $G$, which contradicts the fact that $G$ is irreducible.
\end{proof}

For $(\alpha_{1},\dots, \alpha_n)\in (C^*)^n$, let $$E_n (\alpha_{1},\dots, \alpha_n):=\begin{pmatrix}
& \alpha_1&&\\ 
 &  &\ddots & \\
&&&\alpha_{n-1}\\
\alpha_n&&&
\end{pmatrix}=
\mathrm{Diag} (\alpha_{1},\dots, \alpha_n)E_n, $$
where
$E_n=\begin{pmatrix}
& 1&&\\ 
 &  &\ddots & \\
&&&1\\
1&&&
\end{pmatrix}$ is the companion matrix with last row  $(1,0,\dots,0)$. To lighten the notations, we will denote it $E$. For $(\alpha_{1},\dots, \alpha_n)\in (C^*)^n$, let $\bar{\alpha}:=(\alpha_{1},\dots, \alpha_n)$ and let $\Pi\bar{\alpha} :=\prod_{i=1}^{n}\alpha_{i}$.  Note that $\left(E_n (\alpha_{1},\dots, \alpha_n)\right)^n=(\Pi \bar{\alpha} )\mathrm{Id}_n$. Let $\mathcal{M}_\ell$ be the set of invertible matrices of the form $\mathrm{Diag} (\alpha_{1},\dots, \alpha_n)E^{\ell}$.  Note that $\mathcal{M}_0$ is composed of diagonal matrices. The following theorem gives information on the possible groups that may occur as Galois groups.

\begin{thm}\label{thm1}
Assume that $n$ is a prime number.  Let $G \subset \mathrm{GL}_n (C)$ be an imprimitive irreducible group such that $G/G^{0}$ is finite and cyclic. Then, up to a conjugation,  one of the following holds:
\begin{enumerate}
 \item\label{cas1} $G= \displaystyle \bigcup_{\ell=0}^{n-1} \mathcal{M}_{\ell}$. In that case $G^0\sim (C^*)^n$,  
$G/G^0 \sim  \Z/ n \Z$  and $\det(G)=C^*$.
 \item\label{cas2} There exists a positive integer $s$ such that 
 $$G\cap \mathcal{M}_0= \left\{\mathrm{Diag}(\alpha_{1},\dots, \alpha_n)\mid  (\Pi\overline{\alpha})^{s}=1\right\}$$ and there exists $\mathscr{M}$ of the form $E_n\left(\alpha_{1},\dots, \alpha_n\right)$, where $\Pi \overline{\alpha}=e^{2\mathbf{i}\pi k/(ns)}$ with $\gcd (k,s)=1$, such that $G/G^0$ is generated by $\mathscr{M}$.  In that case,  we find  that 
 ${G^0=  \left\{\mathrm{Diag}(\alpha_{1},\dots, \alpha_n)\mid \Pi\overline{\alpha}=1\right\}\sim (C^*)^{n-1}}$, 
$G/G^0 \sim  \Z/ ns \Z$  and 
 \begin{enumerate}
\item\label{cas21}   if $n=2$ then  $\det(G)=\Z/  (2s/\gcd(k+s,2)) \Z$;
\item\label{cas22}  if $n> 2$ then   $\det(G)=\Z/  (ns/\gcd(k,n)) \Z$.
 \end{enumerate}
\end{enumerate}
\end{thm}

\begin{proof}
Up to a conjugation, $G$ is a subgroup of $ \displaystyle \bigcup_{\ell=1}^{n} \mathcal{M}_{\ell}$. 
Since $G$ is irreducible, it contains a nondiagonal element thus there exists $0<\ell<n$ such that $G$ contains an element of $\mathcal{M}_{\ell}$.  Using that $G$ is a group, $G$ contains an element of $\mathcal{M}_{k\ell}$ for all $k\in\mathbb{Z}$. Since $n$ is prime,   for all $0\leq k'<n$, there exists an integer $k$ such that $k\ell= k' \mod n$. This shows that $G$ contains an element of $\mathcal{M}_{k'}$, for all $0\leq k' <n$. 
Let $\mathcal{D}$ be the diagonal elements of $G$.  It induces an algebraic subgroup of $n$ copies of the multiplicative group.  

Assume first that $\mathcal{D}$ is equal to the group of invertible diagonal matrices.  Since the latter is connected,   $G^0=\mathcal{D}$. 
Thus, $G$ contains all the invertible diagonal elements. Since $G$ contains an element of $\mathcal{M}_{k'}$ for all $0\leq k' <n$ and contains $\mathcal{M}_0$ it follows that $G= \displaystyle \bigcup_{\ell=0}^{n-1} \mathcal{M}_{\ell}$.  This is the first case. 

\medskip
 
Assume now that $\mathcal{D}$ is strictly included in the group of invertible diagonal matrices. Since $G\subset  \displaystyle \bigcup_{\ell=1}^{n} \mathcal{M}_{\ell}$, we have $G^0\subset \mathcal{D}$. There exists $\bn:=(s_1,\dots s_n)\in\mathbb{Z}^n$, where the $s_i$ are not all zero, such that the diagonal matrices $\mathrm{Diag}(\alpha_{1},\dots, \alpha_n)$ in $G$ satisfy $\prod\limits_{j=1}^n \alpha_j^{s_j}=1$.   We have $\sum_{j=1}^{n}s_j \ln(|\alpha_j|)=0$. 

Consider a matrix $P:=E_n\left(\beta_1,\ldots, \beta_n\right)\in \mathcal{M}_1\cap G\, .$
We have 
$$P .\text{Diag}\left(\alpha_n,\alpha_1,\alpha_{2},\ldots, \alpha_{n-1} \right) = \text{Diag}\left(\alpha_1,\alpha_{2},\ldots, \alpha_{n} \right).P\, .$$
Therefore, if $\text{Diag}\left(\alpha_1,\alpha_{2},\ldots, \alpha_{n} \right)\in G$ then $\text{Diag}\left(\alpha_n,\alpha_1,\alpha_{2},\ldots, \alpha_{n-1} \right)\in G$.
Iterating this computation,  we obtain that the nonzero vector $\mathfrak{X}=( \ln(|\alpha_{1}|),\dots,  \ln(|\alpha_{n}|))^t$ is a solution of 
$$B\mathfrak{X}=0, \quad  \hbox{ where} \quad B=\begin{pmatrix}
s_1&s_2 &\cdots &s_n\\ 
s_n &s_1  &\ddots &\vdots \\
\vdots &\ddots & \ddots & s_2\\
s_{2}&\cdots & s_n &s_{1}
\end{pmatrix}.$$
Note that $B= \sum_{\ell=0}^{n-1} s_{\ell+1} E^{\ell}$ is a circulant matrix.  Since   $E^n$ is the identity, $E$ is diagonalisable with eigenvectors $\mathfrak{X}_k=(1, e^{2\mathbf{i}k\pi/n}, \dots, e^{2\mathbf{i}k(n-1)\pi/n} )^t$, $0\leq k\leq n-1$, and its eigenvalues are the $n$th roots of unity:
$$
E\mathfrak{X}_k = e^{2\mathbf{i}k\pi/n}\mathfrak{X}_k\, .
$$
These eigenvectors are also eigenvectors for every polynomial in $E$,  in particular they are eigenvectors of $B$.
Since $E$ is diagonalisable, there exists $P\in\GL_n(C)$  such that $E=PD_0 P^{-1}$, where $D_0 =\mathrm{Diag}(1,e^{2\mathbf{i}\pi/n},\dots,e^{2\mathbf{i}(n-1)\pi/n})$ is a diagonal matrix. 
Since $B\mathfrak{X}=0$ and $P^{-1}BP=\sum_{\ell=0}^{n-1} s_{\ell+1} D_0^{\ell}$, the diagonal matrix $\sum_{\ell=0}^{n-1} s_{\ell+1} D_0^{\ell}$ has a nontrivial kernel, that is, one of its diagonal entries is zero: there exists $0\leq k_0 \leq n-1$ such that  $\sum_{\ell=0}^{n-1} s_{\ell+1} e^{2\mathbf{i}\ell k_0\pi/n}=0$. 
The $n$th roots of unity satisfy $z^n=1$ and we have $z^n -1=(z-1)\Phi_n(z)$ where $\Phi_n(z):=1+z+\dots +z^{n-1}$.  Because $n$ is prime, the cyclotomic polynomial $\Phi_n$ is irreducible on $\mathbb{Z}$.   If $k_0>0$, then we must have $s:=s_1=\dots=s_n\neq 0$. Let us prove by contradiction that we work in the case $k_0\neq 0$.  

By contradiction, assume that  $k_0=0$.  Then,  $\sum_{\ell=0}^{n-1} s_{\ell+1}=0$.  Recall that the $s_i$ are not all zero.  In particular,  we cannot be in the situation where $s_1=\dots=s_n$. From what precedes,  this implies that for all  $k\neq0$  we have $\mathfrak{X}_k\notin\ker(B)$.  Thus, $\ker(B)=\mathrm{Vect}(\mathfrak{X}_0)$ and in particular $\mathfrak{X}\in \mathrm{Vect}(\mathfrak{X}_0)$, that is $\ln(|\alpha_{1}|)=\dots=\ln(|\alpha_{n}|)$. We obtain 
\begin{equation}
\label{eq:mod}
|\alpha_1|=\ldots = |\alpha_n|\, .    
\end{equation}  Let $g:=\gcd(s_1,\ldots,s_n)$, $s_i':=\frac{s_i}{g}$ and $\beta_i := \alpha_i^g$ for $i\in \lbrace 1,\ldots, n\rbrace$. Since $\prod\limits_{j=1}^n \alpha_j^{s_j}=1$, we have $\prod\limits_{j=1}^n \beta_j^{s_j'}=1$. Thus,  there exists $\ell_0\in\mathbb Z$ such that $\sum_{j=1}^{n}s_j' \arg(\beta_j) = 2\pi \ell_0$. Moreover, since $\gcd(s_1',\ldots,s_n') = 1$, by the Bézout's identity, there exist $u_1,\ldots, u_n\in\mathbb{Z}$ such that  $\sum_{j=1}^{n}u_js_j' = 1$.  We can replace $\arg(\beta_j)$ with $\arg(\beta_j) - 2\pi u_j\ell_0$ (new choices for the determinations of arguments), also called $\arg(\beta_j)$ to lighten notations. With these new choices, $\sum_{j=1}^{n}s_j' \arg(\beta_j) = 0$. Thus $\mathfrak{Y}:=(\arg(\beta_1),\ldots,\arg(\beta_n))^t$ satisfies $B\mathfrak{Y}=0$. As before (with $\mathfrak Y$ instead of $\mathfrak X$), we obtain 
\begin{equation}
\label{eq:arg}
\arg(\beta_1)=\ldots=\arg(\beta_n)\, .    
\end{equation} 
From \eqref{eq:mod} and \eqref{eq:arg}, we have
$
\alpha_1^g = \ldots = \alpha_n^g \, .
$
 Thus, in that case, there exists $g\in \mathbb{N}^*$ such that $\mathcal{D}\subset \lbrace \mathrm{Diag}(\alpha_1,\ldots,\alpha_n) \in\GL_n\left( C\right)\mid \alpha_1^g = \ldots = \alpha_n^g\rbrace$. Let us prove that  the connected component of the identity of $\mathcal{D}$ is composed by dilatations.  Indeed,  with the equality $\alpha_1^g-\alpha_j^g=(\alpha_1-\alpha_j)(\alpha_1^{g-1}+\alpha_1^{g-2}\alpha_j + ... + \alpha_1 \alpha_j^{g-2}+ \alpha_j^{g-1})=0$, we find that $\lbrace \mathrm{Diag}(\alpha_1,\ldots,\alpha_n) \in\GL_n\left( C\right)\mid \alpha_1^g = \ldots = \alpha_n^g\rbrace$ contains the closed set $\alpha_1=\ldots = \alpha_n$. The latter is composed by dilatations matrices and contains  the connected component of the identity. This implies that the elements of $G^0$ are dilatations, which is a contradiction because of Lemma \ref{lem:irred_primitive_G} (with $H:=G^0$ in this lemma).  Therefore, $k_0>0$ so we must have $s:=s_1=\dots=s_n\neq 0$.
\color{black}

We proved that $\mathcal{D}\subset \mathcal{D}_s:=\lbrace \mathrm{Diag}(\alpha_{1},\dots, \alpha_n)\mid (\Pi \overline{\alpha})^{s}=1\rbrace$ and, since $\mathcal{D}_{s}=\mathcal{D}_{-s}$, we can assume that $s$ is positive. If $\mathcal{D}$ is a strict subgroup of $\mathcal{D}_s$, the same reasoning implies that $\mathcal{D}\subset \mathcal{D}_{s'}$ for some $s'$ that divides $s$. So, up to taking a smaller positive integer $s$, we may assume that $\mathcal{D}= \mathcal{D}_s$.
In this case, $G^0=\mathcal{D}_1$.  
 
Let $\mathscr{F}:=\mathrm{Diag} (\alpha_{1},\dots, \alpha_n)E^{\ell}\in \mathcal{M}_{\ell}\cap G$ be a non diagonal element of $G$ which generates the finite group $G/G^0$ (we can choose in particular $0<\ell<n$). Since $n$ is prime, $n$ is the smallest integer such that $\mathscr{F}^n$ is diagonal. We have $\mathscr{F}^n = \mathrm{Diag} (\Pi \baralp ,\dots, \Pi \baralp)$ thus it satisfies $(\Pi \baralp)^{sn}=1$. We have $\Pi \baralp=e^{2\mathbf{i}\pi k/sn}$ for some integer $0\leq k<sn$ and $\mathscr{F}^n$ has to generate the diagonal elements of $G$. This shows that $(\Pi \baralp)^n$ has to be a primitive $s$th root of unity, that is, $k$ has to be coprime with $s$. Since $G$ is defined up to an isomorphism and since $\mathscr{F}$ is conjugate to $\mathrm{Diag} (\alpha_{\mu(1)},\dots, \alpha_{\mu(n)})E$ where $\mu\in\mathfrak{S}_n$ is a permutation, we can assume (up to renumbering the $\alpha_i$'s, which can be done since we just have conditions on $\Pi \baralp$) that $\mathscr{F} = \mathrm{Diag} (\alpha_{1},\dots, \alpha_n)E = E_n\left(\alpha_{1},\dots, \alpha_n\right)$.\par 
The determinant of the diagonal elements is a power of $e^{2\mathbf{i}\pi/s}$. When $n>2$, $n$ is odd and the determinant of $\mathscr{F}$ is $e^{2\mathbf{i}\pi k/sn}$. When $n=2$ the determinant of $\mathscr{F}$ is  $-e^{2\mathbf{i}\pi k/2s}=e^{2\mathbf{i}\pi (k+s)/2s}$.   
The results on the determinant  follow.
\end{proof}

\begin{rem}
When $n=2$, we recover \cite[Lemma 4.8]{hendriks1998algorithm}.
\end{rem}

Let us now see how $\det(G)$ helps us to distinguish the cases in Theorem \ref{thm1} and compute~$s$. We will explain how to compute $\det(G)$ in Section \ref{secorder1}. 

\begin{prop}
\label{prop:detDistinguish}
Assume that $n$ is prime.  Let $G \subset \mathrm{GL}_n (C)$ be an imprimitive irreducible group such that $G/G^{0}$ is finite and cyclic.  The following statements holds.
\begin{itemize}
\item  $\det(G)=C^*$ if and only if  $G$ is the group  in the case \ref{cas1} of Theorem \ref{thm1}.   
\item If  $\det(G)=\Z/  \nu \Z$, with $\nu \in n\N^*$,  then $s=\nu/n$. where $s$ is defined in  Theorem~\ref{thm1}.
\item Otherwise,  $s=\nu$. 
\end{itemize}
\end{prop}

\begin{proof}
If $\det(G)=C^*$ we are clearly in the case case \ref{cas1} of Theorem \ref{thm1}.  The converse is also true.  
 \par 
 Assume that $n\neq 2$, thus by Theorem \ref{thm1},  $\det(G)=\Z/  (sn/\gcd(k,n)) \Z$.  We know the integers $n$, $\nu:=(sn/\gcd(k,n))$  and we want to compute $s$. 
Note that, since $n$ is prime,  $\nu=sn/\gcd(k,n)$ is either $sn$ or $s$.  Let us prove that $\nu=s$ if and only if  $\gcd(\nu,n)=1$.    If $\nu=s$ then $\gcd(k,n)=n$ and, since  $k$ is coprime with $s$,  this implies $\gcd(n,s)=1$, that is $\gcd(\nu,n)=1$. Obviously, if $\nu=sn$ then $\gcd(\nu,n)=n$.  
Therefore, either $\gcd(\nu,n)=1$ and $s=\nu$ or $\gcd(\nu,n)=n$ and $s=\nu/n$. \par 
In the case $n=2$, $k$ is replaced by $k+s$.  Note that $k+s$ is coprime with $s$ and the same conclusion holds.  
\end{proof}

\subsubsection{How to compute a reduced form?}

First, we prove a result which is useful in what follows.

\begin{prop}
\label{prop:irred_primitive}
    Let $n$ be a prime number and let $H\subset\GL_n(C)$ be an irreducible group. Assume that $K$ is not composed only of dilatations and $K\triangleleft H$. If $H$ is a primitive group, then $K$ is an irreducible group.
\end{prop}

\begin{proof}
Let us prove by contradiction the irreducibility of $K$. If $K$ is not irreducible, there exists a vector space $V\subsetneq C^n$ of minimal dimension $\nu>0$ such that $K(V)=V$.  
We denote by $V_1,\dots, V_k$  vector subspaces of $C^n$ in direct sum such that $\nu=\dim(V_j)$ and $K(V_j)=V_j$.  We then have $V_1\oplus \dots \oplus V_k\subset C^n$. Thereofore $\nu k< n$ and $k$ is bounded. Then, without loss of generalities we may assume that $k$ is maximal, meaning that for all $V'\subset C^n$, with $\dim (V')=\nu$ and $K(V')=V'$, $V'$ is not in direct sum with $V_1\oplus \dots \oplus V_k$.

We claim that  $V_1\oplus \dots \oplus V_k =C^n$.  To the contrary, assume that $V_1\oplus \dots \oplus V_k \subsetneq C^n$. By irreducibility of $H$, there exist $\mathscr{H}\in H$ and $\ell\in\lbrace 1,\ldots,k\rbrace$ such that $Y:=\mathscr{H}(V_\ell)$, whose dimension is $\nu$, is not included in $V_1\oplus \dots \oplus V_k$. Since $K$ is a normal subgroup of $H$, $K\mathscr{H}=\mathscr{H}K$. In particular, $KY=Y$ so $Y$ is stable under $K$.  
Since $Y\cap (V_1\oplus \dots \oplus V_k)$ is stable under $K$, by minimality of the dimension $\nu$, it follows that the dimension of $Y\cap (V_1\oplus \dots \oplus V_k)$ is either $0$ or $\nu$. In the first case, this proves that $Y$ is in direct sum with $V_1\oplus \dots \oplus V_k$ and, in 
the second case, this proves that $Y \subset V_1\oplus \dots \oplus V_k$. In both situations this is a contradiction thus $V_1\oplus \dots \oplus V_k= C^n$. Hence $\nu k=n$ and, since $n$ is prime, we have $k=n$ and $\nu=1$. 
At this stage, we know that $V_1\oplus \dots \oplus V_n= C^n$ with $\dim(V_j)=1$ and $K(V_j)=V_j$ for all $j\in\lbrace 1,\ldots,n\rbrace$. This implies that $K$ is composed of diagonal matrices in a basis $\mathcal{B}$ induced by $V_1,\ldots,V_n$. 

Let us prove by contradiction that the common eigenspaces of the matrices in $K$ must have the same dimension. Indeed, 
if it is not the case and if $W_1,\ldots,W_\ell$ are the eigenspaces of $K$ of maximal dimension, then we have $W:=W_1\oplus\ldots\oplus W_\ell\subsetneq C^n$ (it is not $C^n$ because we assume that there exists at least an eigenspace which has an another dimension). Since $K(W_i)=W_i$  for all $i\in\lbrace 1,\ldots,n\rbrace$ and  $K\mathscr{H}=\mathscr{H}K$ for all $\mathscr{H}\in H$, we have $K(\mathscr{H}(W_i))=\mathscr{H}(W_i)$. Thus, by maximality of $\dim(W_i)=\dim(\mathscr{H}(W_i))$, $\mathscr{H}(W_i)$ is a common eigenspace of the matrices in $K$, that is $\mathscr{H}(W_i)=W_j$ for a certain $j$. Hence  $\mathscr{H}(W_i)\subset W$, which proves that $H(W)\subset W$ (and taking the dimension, we have $H(W)=W$), in contradiction with the irreducibility of $H$. 

We proved that the matrices in $K$ are diagonal and we see that the common eigenspaces of the matrices in $K$ have the same dimension. Since, the direct sum of these eigenspaces is $C^n$, this dimension must divides $n$, which is a prime number. Thus, this dimension is $1$ or $n$. In the second case, it means that the matrices in $K$ are dilatations, a contradiction with a hypothesis on $K$. In the first case, the eigenspaces are $V_1,\ldots,V_n$. As before, if $\mathscr{H}\in H$, for all $i\in\lbrace 1,\ldots,n\rbrace$, we have $K(\mathscr{H}(V_i))=\mathscr{H}(V_i)$. Thus, $\mathscr{H}(V_i)$, whose  dimension is $1$, is included in an eigenspace of $K$ so  $\mathscr{H}(V_i)=V_j$ for a $j\in\lbrace 1,\ldots,n\rbrace$. We obtain in this case that any element of $H$ induces a permutation between the $V_i$, which contradicts the fact that $H$ is primitive. In both situations, we obtain a contradiction. Finally, $K$ is irreducible. 
\end{proof}

The rest of the subsection is devoted to the computation of the reduced form.
We recall that $A_{[\ell]}:=\phi^{\ell-1}(A)\dots A$ and $G_{[\ell]}$ is the difference Galois group of the system $\phi^{\ell}(Y)=A_{[\ell]}Y$.  Let us first give a criterion of imprimitivity.

  \begin{prop}\label{prop1}
 Assume that $G$ is irreducible and $n$ prime.  It is imprimitive if and only if $G_{[n]}$ is  diagonal. \par 
 Furthermore,  when   $G$ is imprimitive,   the Riccati equation corresponding to $\phi^n (Y)=A_{[n]}Y$ admits a solution    $d\in \k^*$ and we may assume that the $(1,1)$ entry of $A_{[n]}$ is $d$.  
 \end{prop} 
 
 \begin{proof}
 If $G$ is imprimitive but irreducible,  there exists a gauge transformation $T$ such that $B:=\phi(T)AT^{-1}$ is an invertible matrix of the form  $\begin{pmatrix}
& b_1&&\\ 
 &  &\ddots & \\
&&&b_{n-1}\\
b_{n}&&&
\end{pmatrix}$, where $b_i\in \mathbf{k}^*$,  and  such that $[B]$ is in reduced form.  Then,  $\phi^{n}(T)A_{[n]}T^{-1}=\phi^{n-1}(B)\dots \phi(B)B=:B_{[n]}$ is diagonal.   By Theorem \ref{theo:reduced form}, 
this proves that the Galois group of $\phi^{n}(Y)=A_{[n]}Y$ is  diagonal. \par 
    Let us prove the converse.   Assume that the Galois group of $\phi^{n}(Y)=A_{[n]}Y$ is  diagonal,  and $G$ is irreducible.  Consider  a basis of $C^n$ such that the matrices of  $G_{[n]}$ are diagonal. Then, from Lemma \ref{lem2},   $G^0 =G_{[n]}^0$ thus $G^0$ is  diagonal.  Let us prove that $G$ is imprimitive. 
From Lemma \ref{lem:irred_primitive_G} (where $H:=G^0$), the group $G^0$ is not composed only of dilatations. Thus, by Proposition \ref{prop:irred_primitive} (applied to $H:=G$ and $K:=G^0$), if $G$ is a primitive group then $G^0$ is irreducible, which contradicts the fact that $G^0$ is diagonal. Thus, $G$ is imprimitive. 
 
The statement on the Riccati equation is a consequence of Remark~\ref{rem:rightFactorOrder1}.
\end{proof}

Assume that $G$ is irreducible,  imprimitive and $n$ is prime.  By Proposition  \ref{prop1}, the Riccati equation corresponding to $\phi^n (Y)=A_{[n]}Y$ admits a solution    $d\in \k^*$.   Assume that we are able to compute a reduced form of $\phi(y)=dy$, that is compute $f\in \k^*$ such that $\phi(y)=d\frac{\phi(f)}{f}y$ is reduced. The goal of what follows is to prove the following result.

\begin{prop}\label{prop2}
 Assume that $G$ is irreducible, imprimitive  and $n$ prime. The system $\phi(Y)=AY$ is equivalent to $\phi(Y)=E_n (1,\dots, 1, d)Y$. 
The system $\phi(Y)=E_n (1,\dots, 1,f^{-1}, \phi(f)d)Y$ is a reduced form of $\phi(Y)=E_n (1,\dots, 1, d)Y$ with reduction matrix $\mathrm{Diag}(1,\ldots,1,f)$.
\end{prop}

Let us begin by proving a technical lemma. 
\begin{lem}\label{lem5}
 Assume that $G$ is irreducible, imprimitive  and $n$ prime.  Then, there exist $T\in \mathrm{GL}_n (\mathbf{k})$ and $u\in \mathbf{k}^*$ such that $\phi(T)AT^{-1}=\begin{pmatrix}
& 1&&\\ 
 &  &\ddots & \\
&&&1\\
u&&&
\end{pmatrix}$, that is, $\phi(T)AT^{-1}=E_n(1,\ldots,1,u)$. 
\end{lem}

\begin{proof}
This is inspired by \cite[Section 4.2]{hendriks1998algorithm}. By Theorem \ref{theo:reduced form}, there exist $T\in \mathrm{GL}_n (\mathbf{k})$ and $u_1,\dots,  u_n \in \mathbf{k}^*$  such that 
$\phi(T)AT^{-1}=E_n(u_2,\ldots,u_n,u_1)$.
Let $D:=\mathrm{Diag}\left(a_1,\ldots,a_n\right)$.
We have the following equality 
$$
\phi(D)E_n(u_2,\ldots,u_n,u_1)D^{-1} = E_n\left(\frac{\phi(a_1) u_2}{a_2},\ldots,\frac{\phi(a_{n-1}) u_n}{a_n},\frac{\phi(a_n) u_1}{a_1}\right)\, .
$$

If we perform the gauge transformation $D:=\mathrm{Diag}(1,u_2,1,\ldots,1)$, we may reduce to the case $u_2=1$. Then,  if $n>2$,  performing the gauge transformation $D:=\mathrm{Diag}(1,1,\phi(u_2)u_3,1,\ldots,1)$, we may reduce to the case $u_3=1$. Iterating this process  allows at the end to reduce to the case $u_2= \dots=u_n=1$ and we change $u_1$ to $u:=u_1\phi(u_n)\phi^2(u_{n-1})\ldots\phi^{n-1}(u_2)$.
\end{proof}

Let us continue the proof of Propositon \ref{prop2}. 
By the previous lemma, we have that $[A]$ is equivalent to $E_n(1,\ldots,1,u)$, therefore $\phi^n(Y)=A_{[n]}Y$ is equivalent to $\phi^n(Y)=\mathrm{Diag}(u,\phi(u),\dots,  \phi^{n-1} u)Y$.  Unfortunately, we are not yet able to compute $u$.   By Proposition \ref{prop1},  there exists 
 a diagonal system  $\phi^n (Y)=DY$, that is  equivalent to $\phi^n (Y)=A_{[n]}Y$ with $(1,1)$ entry $d$. 
 We write 
$$
D:=\mathrm{Diag}(d_1,d_2,d_3,\ldots,d_n)
$$
where we set $d_1:=d$.  Let us now see that we may reduce to the case where $d_i=\phi^{i-1-i_1}(d)$ for an integer $i_1\in\lbrace 1,\ldots,n\rbrace$. 
\begin{lem}
There exists a  gauge transformation $X$ and an integer $i_1\in\lbrace 1,\ldots,n\rbrace$ such that 
\begin{equation}
\label{eq:findX}
   \phi^n(X)D X^{-1} = \mathrm{Diag}(\phi^{-i_1}(d),\ldots,\phi^{n-1-i_1}(d))\, . 
\end{equation}
\end{lem}

\begin{proof}
Note that there exists  an invertible matrix $Z$ such that   $\phi^n (Z)DZ^{-1}=\mathrm{Diag}(u,\phi(u),\dots,  \phi^{n-1} u)$. 
Let $Z_{i,j}$ be the entries of $Z$. Since $Z$ is invertible, for all $j$ there exists $i_j\in\lbrace 1, \ldots,n\rbrace$ such that $Z_{i_j,j}\neq 0$. We can moreover assume that the $i_j$ are distinct. Taking the $(i_j,j)$-coefficient in $\phi^n (Z)D=\mathrm{Diag}(u,\phi(u),\dots,  \phi^{n-1} (u))Z$ yields $\phi^n  (Z_{i_j,j})d_j=\phi^{i_j-1}(u)Z_{i_j,j}$. Thus, for all $j\in\lbrace 1,\ldots, n\rbrace$, there exist $i_j\in\lbrace 1,\ldots, n\rbrace$ and $f_j\in\k^*$ such that 
\begin{equation}
    \label{eq:d}
    \phi^n(f_j)d_j = \phi^{i_j-1}(u)f_j\, .
\end{equation}
 For $j=1$ in \eqref{eq:d}, we have $\phi^{i_1-1}(u) = \frac{\phi^n(f_1)d_1}{f_1}$ thus $\phi^{i_j-1}(u) = \frac{\phi^n(\phi^{i_j-i_1}(f_1))\phi^{i_j-i_1}(d_1)}{\phi^{i_j-i_1}(f_1)}$. Therefore, from \eqref{eq:d}, setting $g_j:= \frac{f_j}{\phi^{i_j-i_1}(f_1)}$,  we obtain
 $$
 \phi^n(g_j)d_j= \frac{\phi^n (f_j)d_j}{\phi^n(\phi^{i_j-i_1}(f_1))}= \frac{ \phi^{i_j-1}(u)f_j}{\phi^n(\phi^{i_j-i_1}(f_1))}=  \phi^{i_j-i_1}(d_1)g_j.
 $$ 
Applying $\phi^{i_1}$ to this equality, we proved that there exists an invertible diagonal matrix $Z_0$ and distinct integers $i_1,\ldots,i_n\in\lbrace 1,\ldots,n\rbrace$ such that 
$$
\phi^n(Z_0)\phi^{i_1}(D)Z_0^{-1} = \mathrm{Diag}(\phi^{i_1}(d_1) ,\dots,  \phi^{i_n}(d_1))\, . 
$$
Thus there exists a permutation matrix $Z_1$ such that 
$$
\phi^n(Y)\phi^{i_1}(D) Y^{-1} = \mathrm{Diag}(d,\ldots,\phi^{n-1}(d))\quad \text{with } Y:=Z_1Z_0\, .
$$
 Thus,
$$
\phi^n(X)D X^{-1} = \mathrm{Diag}(\phi^{-i_1}(d),\ldots,\phi^{n-1-i_1}(d))\quad \text{where }X:=\phi^{-i_1}(Y)\, .
$$
\end{proof}

Thus, at this stage, we know  the existence of $i_1\in\lbrace 1,\ldots,n\rbrace$  such that   $\phi^n (Y)=A_{[n]}Y$ is equivalent to 
$$
\phi^n(Y)=\mathrm{Diag}(\phi^{-i_1}(d),\phi^{1-i_1}(d),\ldots,  \phi^{n-1-i_1}(d))Y\, .
$$
Since the system 
$
\phi^n(Y)=\mathrm{Diag}(\phi^{-i_1}(d),\phi^{1-i_1}(d),\ldots,  \phi^{n-1-i_1}(d))Y
$ is diagonal, we may assume that we have a diagonal fundamental matrix of solutions $U$. This matrix can be taken of the form $U=\mathrm{Diag}(y,\phi(y),\dots,  \phi^{n-1}(y))$. Then, 
$$
\phi(U)=FU, \quad \text{where}\quad F:=\begin{pmatrix}
& 1&&\\ 
 &  &\ddots & \\
&&&1 \\
\phi^{-i_1}(d)&&&
\end{pmatrix}=E_n (1,\dots, 1, \phi^{-i_1}(d))\, .
$$
By construction, 
$F_{[n]} = \mathrm{Diag}(\phi^{-i_1}(d),\phi^{1-i_1}(d),\ldots,\phi^{n-1-i_1}(d))$. We thus have two $\phi^n$- systems $\phi^n(Y) = A_{[n]}Y$ and $\phi^n(Y) = F_{[n]}Y$, that are equivalent and we want to prove that they are equivalent as $\phi$-system. This is the goal of the next lemma. 

\begin{lem}
\label{lem:gaugeTransfoEquiv}
Let $F\in\GL_n(\k)$ be such that the systems $\phi^n(Y) = F_{[n]}Y$ and $\phi^n(Y) = A_{[n]}Y$ are equivalent. 
Then, the systems $[A]$ and $[F]$ are equivalent.
\end{lem}

\begin{proof}
Let $\mathcal{Y}$ and $\mathcal{Z}$ be two fundamental matrices of solutions of $\phi(Y)=AY$ and $\phi(Y)=FY$ respectively. 
Thus, $\phi^n(\mathcal{Y})=A_{[n]}\mathcal{Y}$ and $\phi^n(\mathcal{Z})=F_{[n]}\mathcal{Z}$. By hypothesis,  there exists $T\in\GL_n(\k)$ such that $\mathcal{X}:=T\mathcal{Y}$ is a solution of $\phi^n(Y)=F_{[n]}Y$.
This implies that  $P:=\mathcal{X}^{-1}\mathcal{Z}$ is a $\phi^n$-constant  matrix.  By \cite[Lemma 3.1]{DH21}, the entries of $P$ belong to $C$ so we have 
$$
A = \phi(\mathcal{Y})\mathcal{Y}^{-1} = \phi(T^{-1}) \phi(\mathcal{X})\mathcal{X}^{-1}T= \phi(T^{-1}) \phi(\mathcal{Z})\mathcal{Z}^{-1} T =\phi(T)^{-1} FT .
$$
\end{proof}

Let us finish the proof of Proposition \ref{prop2}.  From what precedes,   $[A]$ is equivalent to $[F]$ where 
$F=E_n (1,\dots, 1, \phi^{-i_1}(d))$. Using the following lemma and the fact that $\phi(1)=1$, we obtain that this system is equivalent to $[J]$, where $J:=E_n (1,\dots, 1, d)$.

\begin{lem}
Let $J\in \GL_n(\k)$ and $\ell\in\mathbb{N}$. The systems $[\phi^{\ell}(J)]$ and $[\phi^{\ell +1}(J)]$ are equivalent.
\end{lem}
\begin{proof}
    It follows from the fact that $\phi^{\ell+1}(J) = \phi(X)\phi^{\ell}(J) X^{-1}$ with $X:=\phi^{\ell}(J)$.
\end{proof}

Thus, we have that $[A]$ and $[J]$ are equivalent systems. 
Now, we perform the gauge transformation $S=\mathrm{Diag}(1,\ldots,1,f)$ to find that $[A]$ is equivalent to $\phi(Y)=E_n (1,\dots, 1,f^{-1}, d\phi(f))Y$.  It remains to check that this system is in a reduced form.

By Theorem \ref{theo:reduced form},  there exists a reduction matrix of the form 
 $T:= \mathrm{Diag}(f_1,\dots, f_n)E^k$.  Note that $T$ may be the identity.  We are going to apply Theorem \ref{thm1} and look at the product of the nonzero entries of  $B:=\phi(T)E_n (1,\dots, 1,f^{-1}, d\phi(f))T^{-1}$. The latter  is $\frac{\phi(f_1\times \dots \times f_n) }{f_1\times \dots \times f_n}\frac{\phi(f)}{f}d$.  Since $\phi y=\frac{\phi(f)}{f}dy$ is in reduced form,  the algebraic group generated by $\frac{\phi(f_1\times \dots \times f_n) }{f_1\times \dots \times f_n}\frac{\phi(f)}{f}d$ is bigger than the algebraic group generated by $\frac{\phi(f)}{f}d$. Since $[B]$ is reduced, by Theorem~\ref{theo:reduced form},  it is also smaller, proving that the algebraic group generated by $\frac{\phi(f_1\times \dots \times f_n) }{f_1\times \dots \times f_n}\frac{\phi(f)}{f}d$ is the algebraic group generated by $\frac{\phi(f)}{f}d$.  By Theorem \ref{theo:reduced form} again, this shows that $\phi(Y)=E_n (1,\dots, 1,f^{-1}, d\phi(f))Y$ is already reduced.

\begin{ex}\label{ex2}
Let us go back to the example \eqref{eq0} and recall, see Example \ref{ex1}, that the difference Galois group $G$ is irreducible. Assume that $t\neq 0$. Instead of applying the criterion of Proposition~\ref{prop1} to show that $G$ is imprimitive,  let us use a shortcut.  We refer to \cite[Section 2.2]{ramis2013local} for more details on what follows. 
We define the Newton polygon of a $q$-difference operator 
$L=\sum_i a_{i} \phi^k$
as the convex hull in $\R^2$ of  
$$
\{(i,j) \ \vert \ i,j \in \Z, \ j \geq \val(a_{i}(z))\} \subset \R^{2}
$$
where $\val : \C((z)) \rightarrow \Q \cup \{+\infty\}$ denotes the $z$-adic valuation.  The Newton polygon  is invariant by an analytic change of variables. 
By Lemma \ref{lem5}, when $n$ is prime, if the difference Galois group of an equation is irreducible and imprimitive,  then the lower bound of the Newton polygon is delimited by the segment from $(0,\val (u))$ to $(n,0)$.  In Example \eqref{eq0}, the lower part is composed of the segment delimited by $(0,1)$ and $(1,0)$ and the segment delimited by $(1,0)$ and $(3,0)$, see Figure \ref{fig1}.  Since $G$ is irreducible,  see Example~\ref{ex1}, we find that $G$ is primitive. 

\begin{figure}[!h]
\begin{center}
\begin{tikzpicture}
\coordinate (O) at (0,2);
\coordinate (A) at (0,1);
\coordinate (B) at (1,0);
\coordinate (C) at (2,0);
\coordinate (D) at (3,0);
\coordinate (E) at (3,2);
\draw (O) -- (A);
\draw (A) -- (B);
\draw (B) -- (D);
\draw (E) -- (D);

\draw [dotted] (0,0) grid (3,2);
\put(0,30){{$(0,1)$}}
\put(20,-15){{$(1,0)$}}
\put(50,-15){{$(2,0)$}}
\put(80,-15){{$(3,0)$}}
\node (centre) at (0,1){$\bullet$};
\node (centre) at (1,0){$\bullet$};
\node (centre) at (2,0){$\bullet$};
\node (centre) at (3,0){$\bullet$};
\end{tikzpicture}
\caption{The Newton polygon of the operator associated with $U(z,t)$.}\label{fig1}
\end{center}
\end{figure}
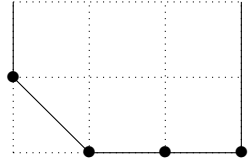

\end{ex}

\subsection{The primitive case} 
In this section, we assume that $G$ is irreducible, primitive, and $n$ is prime. We want to prove that the same holds for $G^0$ and $D(G^0)$ where $D(G^0)$ denotes the derived group of $G^0$. Note that $G^0$ is a connected normal subgroup of $G$ and $D(G^0)$ is a connected normal subgroup of $G^0$. 

First, we recall the following result given in \cite[Lemma 4.2]{adamczewski2021hypertranscendence}.

\begin{lem}
\label{lem:connected+irred}
    If $K\subset \GL_n(C)$ is a connected and irreducible algebraic group, then $K$ is also primitive. 
\end{lem}

\begin{lem}
\label{lem:irred_primitive_G^0}
Let $n$ be a prime number. If $H\subset\GL_n(C)$ is an irreducible and primitive group, then the derived subgroup of $H$, denoted by $D(H)$, is not composed only of dilatations.
\end{lem}

\begin{proof}
By contradiction, assume that $D(H)$ is composed only of dilatations. In this case, for all $\mathscr{G},\mathscr{H}\in H$, there exists $\alpha\in C^*$ such that $\mathscr{G}\mathscr{H}\mathscr{G}^{-1}\mathscr{H}^{-1} = \alpha \textnormal{Id}_n$, that is  $\mathscr{G}\mathscr{H}=\alpha \mathscr{H}\mathscr{G}$. The group $H$ is irreducible so there exists $\mathscr{H}_0\in H$ having an eigenspace $W$ of maximal dimension $\nu<n$, otherwise $H$ would be composed of dilatations contradicting its irreducibility. Let $V$ be the direct sum of the eigenspaces of $\mathscr{H}_0$ of dimension $\nu$. From the previous equality and the fact that $\mathscr{H}_0(W)=W$, we obtain that 
\begin{equation}
    \label{eq:permut_eigenspace}
    \forall \mathscr{G}\in H,\hspace{0.15cm} \mathscr{H}_0(\mathscr{G}(W))=\mathscr{G}(W)\, .
\end{equation}
 Thus, $\mathscr{G}(W)$ is also  an eigenspace of $\mathscr{H}_0$  of dimension $\nu$ for all $\mathscr{G}\in H$ so $H(W)\subset V$. Therefore, $H(V)\subset V$ (in fact, taking the dimension, it is an equality of vector spaces) which implies that $V=C^n$ by irreducibility of $H$. Hence  $\nu$ divides $n$. Since $n$ is a prime number and $\nu<n$, we have $\nu=1$. This means that $\mathscr{H}_0$ is a diagonal matrix with distinct eigenvalues and we denote by $V_1,\ldots,V_n$ its eigenspaces. However, from \eqref{eq:permut_eigenspace} with $W:=V_1$, we obtain that every $\mathscr{G}\in H$ induces a permutation between the $V_i$, which contradicts the fact that $H$ is primitive.
\end{proof}

\begin{prop}\label{prop4}
    Assume that $n$ is prime and $G \subset \mathrm{GL}_n (C)$ is an irreducible and primitive group. Then, the same holds for $G^0$ and $D(G^0)$ and we may write $G^0=Z(G^0).D(G^0)$ where $Z(G^0)$ denotes the center of $G^0$.
\end{prop}

\begin{proof}
We recall that $G^0$ is a connected normal subgroup of $G$ and $D(G^0)$ is a connected normal subgroup of $G^0$.

From Lemma \ref{lem:irred_primitive_G} (with $H:=G^0$), $G^0$ is not composed only of dilatations. Thus, by Proposition \ref{prop:irred_primitive} (with $H:=G$ and $K:=G^0$), $G^0$ is irreducible. Using Lemma \ref{lem:connected+irred}, $G^0$ is also primitive. 
Then, the irreducibility of $D(G^0)$ follows from Lemma \ref{lem:irred_primitive_G^0} (with $H:=G^0$) and Proposition \ref{prop:irred_primitive} (with $H:=G^0$ and $K:=D(G^0)$). Thanks to Lemma \ref{lem:connected+irred}, $D(G^0)$ is also primitive.

By \cite[Corollary 8.1.6]{springer1998linear}, any connected algebraic group $H$ satisfies $H\sim Z(H). D(H)$, which proves that $G^0=Z(G^0).D(G^0)$. 
\end{proof}

\section{Galois groups of equations of order one and diagonal systems}\label{sec:GaloisGpsOrderOneEtDiag}

\subsection{Galois groups of equations of order one}\label{secorder1}

We consider the equation of order one
\begin{equation}
\label{eq:orderone}
\phi(y)=\alpha y \, \quad \alpha\in \k^*.
\end{equation}

Let us now see how to compute the Galois group $G$ of this difference equation. Recall that since $G\subset \mathrm{GL}_1 (C)$ is an algebraic group, either $G$ is $C^*$, or it is isomorphic to $\Z/\ell\Z$ for some $\ell \in \N^*$.

\begin{prop}\label{propord1}
 We have $G\sim \Z/\ell\Z$ if and only if there exists $\zeta$, a primitive $\ell$th root of unity and  a nonzero solution $f\in \k^*$ of  $\phi(y)=(\zeta/\alpha)y$.  In that situation,  $\zeta=\alpha\frac{\phi(f)}{f}$, and $[\zeta]$ is a reduced form of $[\alpha]$ with a gauge transformation $(f)$.  
\end{prop}

\begin{proof}
Since the Galois group $G$ of \eqref{eq:orderone} is an algebraic subgroup of $C^{\star}$, it is a finite cyclic group or the multiplicative group $C^{\star}$. 
Any equivalent equation of \eqref{eq:orderone} is of the form $\phi(y)=\alpha\frac{\phi(f)}{f}y$ with $f\in\k^{\star}$. 
Therefore, from Theorem \ref{theo:reduced form}, if $G$ is a finite cyclic group of order $\ell\in\mathbb{N}^*$ then there exist $f\in\k^{\star}$ and $\zeta$, a primitive $\ell$th root of unity,  such that $\zeta= \alpha\frac{\phi(f)}{f}$. \par Conversely, assume that there exist $f\in\k^{\star}$ and $\zeta$, a primitive $\ell$th root of unity,  such that $\zeta= \alpha\frac{\phi(f)}{f}$. From Theorem \ref{theo:reduced form}, $G$ is contained in $\Z/\ell\Z$. Then,  
$G\sim \Z/\ell_1\Z$ where $\ell_1$ divides $\ell$. Then, there exist $g\in\k^{\star}$ and $\zeta_1$, a primitive $\ell_1$th root of unity,  such that $\zeta_1= \alpha\frac{\phi(g)}{g}$. This implies $\alpha^{\ell_1}=g^{\ell_1}/\phi(g^{\ell_1})$. Let $h:=(f/g)^{\ell}$. Thus, $\phi(h) = h$ so $h\in C^*$. Since $C$ is algebraically closed, there exists $c\in C^*$ such that $f=cg$ (with $c^{\ell} = h$). Using this equality, we obtain $\zeta = \zeta_1$ and since $\zeta$ is a primitive $\ell$th root of unity, $\ell=\ell_1$.
\end{proof}

\begin{rem}
Since $\det(G)$ is the Galois group of $\phi (y)=\det (A)y$, an equation of order $1$, we may now compute $\det(G)$ if we are able to reduce the  difference equations of order $1$.
\end{rem}

\begin{ex}\label{ex3}
Consider the difference equation \eqref{eq0} and let $G$ be its difference Galois group over $\C(z^{1/*})$. Assume that $t\neq 0$. We want to compute $\det(G)$.  We have to compute the difference Galois group of $\phi(y)=\frac{-q^3 z}{t}y=\alpha y$ over $\C(z^{1/*})$.  By Proposition~\ref{propord1},  $\det(G)$ is finite if and only if there exists $\zeta$, a primitive $\ell$th root of unity, and a nonzero solution $f\in \k^*$ of  $\phi(y)=(\zeta/\alpha)y$.  But the valuation of $\alpha$ at $z=0$ is $1$, proving that $\zeta/\alpha$ must have valuation $-1$. Since for $f\in\k^*$,  $f$ and $\phi(f)$ have the same valuation,  the equation $\phi(f)=(\zeta/\alpha)f$ with $f\in\k^*$, and $\zeta$ a root of unity,  has no solutions.  Therefore,  $\det(G)=\C^*$.
\end{ex}

\subsection{Galois groups of diagonal systems}\label{sec:diagonalCase}

For more details on what follows, we refer to \cite[Section 2.2]{van2006galois}.
Let us consider a diagonal system 
\begin{equation}
\label{eq:diagsyst}
\phi(Y)=\mathrm{Diag}\left(\alpha_1,\ldots,\alpha_n\right) Y 
\end{equation}
where $\mathrm{Diag}\left(\alpha_1,\ldots,\alpha_n\right)\in\GL_n\left(\k\right)$ is the diagonal matrix whose nonzero entry at the line $i$ is $\alpha_i$.  
By Theorem~\ref{theo:reduced form}, the Galois group is an algebraic subgroup of the $n$-dimensional torus $T$.  The matrices of $G$ are of the form $\mathrm{Diag} (c_1, \dots , c_n)$.  We recall that a character on $T$ is a morphism $\chi:T\rightarrow C^{\star}$. It is of the form   $\chi_{\mathbf{m}}(z_1,\ldots,z_n)=z_1^{m_1}\ldots z_n^{m_n}$ with $\mathbf{m}:=\left(m_1,\ldots,m_n\right)\in\mathbb{Z}^n$. Since an algebraic subgroup of $T$ is the intersection of the kernels of some characters on $T$, we have to find the $\mathbf{m}\in\mathbb{Z}^n$ such that $c_1^{m_1}\ldots c_n^{m_n}=1$ for all  $\mathrm{Diag} (c_1, \dots , c_n) \in G$. 
By Theorem~\ref{theo:reduced form}, a gauge transformation which reduces the system will be diagonal, denoted by  $\mathrm{Diag}(f_1,\dots,f_n)$. Thus, the difference Galois group of \eqref{eq:diagsyst} is the intersection of the kernels of the $\chi_{\mathbf{m}}$ for all $\mathbf{m}\in\mathbb{Z}^n$ such that there exist $f_1,\ldots, f_n\in\k^*$ satisfying
\begin{equation}
\label{eq:conditionsSystDiag}
 \left(\frac{\phi(f_1)}{f_1}\alpha_1 \right)^{m_{1}}\times\ldots \times \left(\frac{\phi(f_n)}{f_n} \alpha_n \right)^{m_{n}}= 1\,.   
\end{equation} 
The question of how to solve  \eqref{eq:conditionsSystDiag} depends on the examples (see one example in \cite[Section 2.2]{van2006galois}). In general it is a complicated task.

\section{Galois groups of equations of order two}\label{secord2}
Consider a difference equation of order $2$
$$  \phi^2 (y)+a_1 \phi (y)+a_0 y=0, \quad \text{where}\quad  a_i\in \k\quad \text{and}\quad  a_0  \neq 0.$$
 Let 
 $$
 L :=   \phi^2 +a_1 \phi +a_0\in \k\left< \phi,\phi^{-1}\right>
 $$ 
be the corresponding difference operator and let $G\subset \GL_{2}(C)$ be its Galois group.  We want to compute $G$.  An algorithm in the cases (Q), (S), (E), (M) is presented in \cite{hendriks1997algorithm, hendriks1998algorithm,  dreyfus2015galois,roques2018algebraic} respectively.   

\subsubsection*{Reducible case}
We first decide whether $L$ has a right factor of order one.  This is the case if and only if $G$ is reducible.  Recall, see Lemma \ref{lem:eq left factor},  that it is equivalent to the existence of  a solution $\alpha$ of the Riccati equation $\phi(y)y+a_1y+a_0=0$. \par 
Assume first that $L$ is reducible and let us write $L=(\phi-\beta)(\phi-\alpha)$.  Then, $\alpha$ is a solution of the Riccati equation  $\phi(y)y+a_1y+a_0=0$, and $\beta$ is determined by $\alpha \beta=a_0$.  The difference equation is equivalent to $\phi(Y)=AY$ where $A=\left(\begin{array}{cc}
\alpha &  1\\
0 & \beta
\end{array}\right)
$. 
 We now follow the strategy of Section \ref{sec:reducible case}.  As in the diagonal case, we reduce the diagonal system $\phi(Y)=\left(\begin{array}{cc}
\alpha &  0\\
0 & \beta
\end{array}\right)Y
$ with a diagonal  matrix  $D$.  Let $B:=\phi(D)AD^{-1}$,  with $B=\left(\begin{array}{cc}
\alpha' &  \gamma'\\
0 & \beta'
\end{array}\right)$.   By \cite[Lemma 4.4]{hendriks1998algorithm}, either $[B]$ is already reduced, or a reduced form is $\left( \begin{array}{cc}
\alpha' &  0\\
0 & \beta'
\end{array}\right)$. 
We now have to decide if there exists $t\in \k$ such that $T=\left(\begin{array}{cc}
1 & t\\
0 & 1
\end{array}\right)$ satisfies $\phi(T)B=\left( \begin{array}{cc}
\alpha' &  0\\
0 & \beta'
\end{array}\right)T$.  Note that this is equivalent to 
$$  \gamma'+ \phi(t) \beta'= \alpha' t.$$
So we just have to decide if an equation of order $1$ has a solution $t$ in $\k$ or not.  
If such a $t$ exists,  then the reduction matrix is given by $TD$,  and a reduced form is $\left( \begin{array}{cc}
\alpha' &  0\\
0 & \beta'
\end{array}\right)$. Otherwise,  the reduction matrix is given by $D$ and a reduced form is $\left(\begin{array}{cc}
\alpha' &  \gamma'\\
0 & \beta'
\end{array}\right)$. \par 
\subsubsection*{Imprimitive case} Assume now that $G$ is irreducible.  It is imprimitive if and only if the difference Galois group of $\phi^2 (Y)=A_{[2]}Y$ is diagonal. We use the reducible case (with $\phi$ replaced by $\phi^2$) to check this property.   In that case we are able to compute $d\in \k^*$,  solution of the Riccati equation corresponding to $\phi^2 (Y)=A_{[2]}Y$.  We now follow Proposition~\ref{prop2}.  The system   $[A]$ is equivalent to $\phi (Y)=\left( \begin{array}{cc}
0 & 1\\
d & 0
\end{array}\right) Y$. 
Let us   compute a reduced form of $\phi(y)=dy$, that is compute $f\in \k^*$ such that $\phi(y)=d\frac{\phi(f)}{f}y$ is reduced.  Then, a reduced form is given by 
$\left( \begin{array}{cc}
0 &  f^{-1}\\
\phi(f)d & 0
\end{array}\right) $ and a reduction matrix is given by $\mathrm{Diag}(1,f)$ .  \par 
\subsubsection*{Primitive case}  Assume that $G$  is irreducible and  primitive. By \cite[Section 4.4]{hendriks1998algorithm},  either  $G=\mathrm{GL}_2 (C)$, or there exists an integer $k$ such that 
$$
G = \lbrace M\in\GL_2\left(\mathbb{C}\right) \mid \det(M)^{k}=1 \rbrace .
$$
 Recall that $a_0=\det(A)$ and the algebraic group $\det(G)$ is the difference Galois group of $\phi(y)=a_0 y$, a system of order 1. Let $f\in \k^*$ be such that $\phi(y)=a_0\frac{\phi(f)}{f}y$ is reduced with Galois group $\det(G)$.
Then, we perform, as in \cite[Section 4.4]{hendriks1998algorithm}, the gauge transformation  $T=\mathrm{Diag}(f,1)$ to obtain a system $[B]$.  Since $\det(B)=a_0\frac{\phi(f)}{f}$,   and $\phi(y)=a_0\frac{\phi(f)}{f}y$ is reduced with Galois group $\det(G)$,  we find that $B\in G(\mathbf{k})$.  Then,  $B$ is a reduced form and $T$ is the reduction matrix.  

\section{Galois groups of equations of order three}\label{secord3}

Consider a linear  difference equation of order $3$ 
$$  \phi^3 (y)+ a_2 \phi^2 (y)+a_1 \phi (y)+a_0 y=0, \quad \text{where}\quad  a_i\in \k\quad \text{and}\quad  a_0 \neq 0.$$
 Let 
 $$
 L :=  \phi^3 + a_2 \phi^2 +a_1 \phi +a_0\in \k\left< \phi,\phi^{-1}\right>
 $$ 
be the corresponding difference operator and let $G\subset \GL_{3}(C)$ be its Galois group.  We want to compute $G$. 
 
 \subsection{Reducible case}
 
 We follow the strategy explained in Section \ref{sec:reducible case} to treat the case where $G$ is reducible.  By \cite[Lemma~4.4]{adamczewski2021hypertranscendence}, $G$ is reducible if and only if the operator $L$ admits a right factor.  The right factor can have order one or two.  So $G$ is reducible if and only if the operator admits a right or left factor of order one.   Recall that up to consider the dual, we may reduce to the case where there is a right factor of order $1$.  \par 
Recall, see Lemma \ref{lem:eq left factor},  that the existence of a right factor of  order one is equivalent to the existence of  a solution $\alpha\in\k$ of the Riccati equation $\phi^2 (y)\phi(y)y+a_2 \phi (y)y+a_1y+a_0=0$.   
Assume that $L$ admits a right factor of order one.   Let us write $L$ as 
$\widetilde{L}(\phi -\alpha)$ where $\widetilde{L}:=\phi^2-\beta \phi-\gamma$. 
Then,  the system $[A_L]$ is equivalent to the system 
$$\phi (Y)= \left(\begin{array}{c|c}
\alpha &  1 \hspace*{0.3cm} 0  \\
\hline
0 &  A_2 
\end{array}\right)Y \, , $$
where the matrix $A_2$ corresponds to the equation of order two $\widetilde{L}y=0$ (that is $A_2 = A_{\widetilde{L}}$).  
From the order one and order two cases, we know how to compute $T_1\in \GL_1 (\k)$ and $T_2\in \GL_2 (\k)$ such that the systems whose matrices are $\phi(T_1)\alpha T_1^{-1}$ and $\phi(T_2)A_2 T_2^{-1}$ are in reduced form. Let us consider 
$$A := \begin{pmatrix}
\phi(T_1)&0 \\ 
0 &\phi(T_2)
\end{pmatrix}\left(\begin{array}{c|c}
\alpha &  1 \hspace*{0.3cm} 0  \\
\hline
0 &  A_2 
\end{array}\right)\begin{pmatrix}
T_1&0 \\ 
0 &T_2
\end{pmatrix}^{-1}.$$
This matrix  $A$ is of the form
$\begin{pmatrix}
\phi(T_1)\alpha T_1^{-1}&\star \\ 
0 &\phi(T_2)A_2 T_2^{-1}
\end{pmatrix}$. 
The first step consists in reducing the bloc diagonal system $[D]$ where $D:=\begin{pmatrix}
\phi(T_1)\alpha T_1^{-1}&0 \\ 
0 &\phi(T_2)A_2 T_2^{-1}
\end{pmatrix}$. Let $\Gdiag$ be the Galois group of this system, let $G_1$ be the Galois group of $[\phi(T_1)\alpha T_1^{-1}]$ and let $G_2$ be the Galois group of $[\phi(T_2)A_2 T_2^{-1}]$.  
The algebraic group $\Gdiag$ is an algebraic subgroup of $G_1\times G_2$. We explain the reduction on a case-by-case approach with respect to $G_2$. We will see that in each situation, reducing $[D]$ is equivalent to reducing a diagonal system whose reduction is explained in Section \ref{sec:diagonalCase}. Theorem \ref{thm:reduc_blocdiag} summarize the results we are going to prove for this first step.

  \begin{thm}
  \label{thm:reduc_blocdiag}
$\bullet$ Assume that $G_2$ is composed of diagonal matrices.  We can assume that $D$ is diagonal.  Then, we have to reduce a diagonal system and the reduction is explained in Section \ref{sec:diagonalCase}.\\
$\bullet$  Assume that $G_2$ is reducible and the first case does not hold.  We can assume that $D$ is upper triangular.  Then,   the reduction matrix is the diagonal matrix $T$ that reduces the diagonal  system $[\mathrm{Diag}(a,b,d)]$, where $a,b,d$ are the diagonal entries of $D$.  \\
$\bullet$ Assume that $G_2$ is imprimitive and irreducible.  We can assume that $\phi(T_2)A_2 T_2^{-1}$ is an anti-diagonal matrix.  Then $D_{[2]}$ is diagonal.  The reduction matrix that reduces $[D]$ is the matrix $T$ that reduces the diagonal  $\phi^{2}$-system $\phi^2 (Y)=D_{[2]}Y$. \\
$\bullet$ Assume that $G_2$ is primitive and irreducible. The reduction matrix is given by  $T:=\mathrm{Diag}(r_1,r_2,1)$, where  $\mathrm{Diag}(r_1,r_2)$ is the reduction matrix of the diagonal system $\phi (Y)=\begin{pmatrix}
\phi(T_1)\alpha T_1^{-1}&0\\ 
0 &\det(\phi(T_2)A_2 T_2^{-1})
\end{pmatrix}Y$.  
  \end{thm}

Let us  now prove the theorem. 

\begin{proof}
$\bullet$  Assume that $G_2$ is  formed by diagonal matrices. Then, $D$ is diagonal,  $\Gdiag$ is an algebraic group of diagonal matrices and the reduction is explained in Section \ref{sec:diagonalCase}. 

\medskip
$\bullet$
  Assume that  $G_2$ is reducible but not formed by diagonal matrices.   Up to a conjugation, we may assume that $G_2$ is formed by upper triangular matrices. By the classification of such nondiagonal reducible algebraic groups  that can be found in \cite[Lemma 4.4]{hendriks1998algorithm},  the upper triangular entry can be any  element of $C^*$.  
We can assume that the matrix $D$ is of the form $\begin{pmatrix}
a&0&0 \\ 
0 &b&c \\
0&0&d
\end{pmatrix}$ with $c\neq 0$.  Let us first reduce the diagonal system $[\mathrm{Diag}(a,b,d)]$. This will change the elements $a,b,c,d$ but, simplifying notations, let us denote by $a,b,c,d$ the new ones.  With Remark~\ref{rem1} and Theorem \ref{theo:reduced form},  we are looking for a reduction matrix of the form $T:=\begin{pmatrix}
1&0&0 \\ 
0&1&t \\
0&0&1
\end{pmatrix}$.
We denote by $[D']$ a reduced system, $D' = \begin{pmatrix}
a&0&0 \\ 
0 &b&c' \\
0&0&d
\end{pmatrix} $ for a $c'\in\k$. 
The only non-trivial equality in $\phi(T)D = D'T $ is on the $(2,3)$ entry: 
$$ 
c+  d \phi(t)  =  bt +c'.$$
We have $c'\neq 0$, otherwise the gauge transformation $\begin{pmatrix}
1&t \\
0&1
\end{pmatrix}$
would transform $[\phi(T_2)A_2 T_2^{-1}]$ into a diagonal system, a contradiction. 
By \cite{nguyen2008algebraic}, for all $c'\neq 0$, the algebraic group generated by   $ \begin{pmatrix}
b&c \\
0&d
\end{pmatrix} $  and $ \begin{pmatrix}
b&c' \\
0&d
\end{pmatrix} $ are the same. Since the $(1,1)$ entry of the algebraic group generated by $\begin{pmatrix}
\star&0&0 \\ 
0&\star&\star \\
0&0&\star
\end{pmatrix}$ does not affect the $(2,3)$ entry, we deduce that the algebraic group generated by $\begin{pmatrix}
a&0&0 \\ 
0&b&c\\
0&0&d
\end{pmatrix}$ and 
$\begin{pmatrix}
a&0&0 \\ 
0&b&c'\\
0&0&d
\end{pmatrix}$ are the same so  we may take $t=0$ and $c'=c$ and the system is already reduced. 

\medskip

$\bullet$ When $G_2$ is irreducible and imprimitive,  up to a conjugation, it is of the form 
 $G_2=G_{2,\mathcal{D}}\cup G_{2,\mathcal{A}}$, where $G_{2,\mathcal{D}}$ is formed by diagonal matrices and $G_{2,\mathcal{A}}$ by anti-diagonal matrices. 
We then have that  $\Gdiag$ is a union of diagonal matrices and   matrices of the form $\begin{pmatrix}
\star&0&0 \\ 
0 &0&\star \\
0&\star&0
\end{pmatrix}$.
We can assume that the matrix $D$ is of the form $\begin{pmatrix}
a&0&0 \\ 
0 &0&b \\
0&c&0
\end{pmatrix}$ since otherwise $G_2$ would be reducible.    

\begin{rem}
Note that $D_{[2]}$ is a diagonal matrix.  By Lemma \ref{lem2}, the connected components of the identity of the difference Galois groups of $[D]$ and $[D_{[2]}]$ are the same.  So if we are only interested in computing the connected component of $G$ we may just reduce the diagonal system $[D_{[2]}]$, see Section~\ref{sec:diagonalCase} for more details.
\end{rem}

Assuming we know how to compute the Galois group of a diagonal system,  we are able to compute a reduction  matrix $P\in \mathrm{GL}_{3}(\k)$ such that the system $\phi^2(Y) = \phi^2(P)D_{[2]}P^{-1}Y $ is reduced.  Consider $B:=\phi(P)DP^{-1}$, we have $B_{[2]}=\phi^2(P)D_{[2]}P^{-1}$ thus $[B_{[2]}]$ is reduced.    Let us prove that the system $[B]$ is already reduced.
  Let $H$ be the smallest algebraic group such that $B\in H(\k)$.  We have $\Gdiag\subset H$ (property of Galois groups, see Theorem~\ref{theo:reduced form}).
By Lemma~\ref{lem2},  the Galois group of the $\phi^2$-system $[B_{[2]}]$ and the Galois group of the $\phi$-system $[B]$ have the same connected component.  Since $[D_{[2]}]$  and $[B_{[2]}]$  have the same Galois group, and $[B_{[2]}]$  is reduced  we find 
$H^0=\Gdiag^0$.  We recall that $H$ is the smallest algebraic group such that $B\in H(\k)$. Then, $H/H^0=H/\Gdiag^0$ is generated by a single element, so it is cyclic (and finite since it is an algebraic group,  see \cite[Lemma 3.5]{hardouin2016galois}). 
Since $\Gdiag/\Gdiag^0$ is finite cyclic, 
$$
\Gdiag /\Gdiag^0 \sim  \Z/ m' \Z , \quad 
H/\Gdiag^0 \sim  \Z/ n' \Z,  
$$
for integers $m',n'$.   
 With Theorem \ref{theo:reduced form} and Lemma \ref{lem2}, there exists $k$ such that $B_{[k]}\in  \Gdiag^0 (\k)$.  
 
Since, $\Gdiag\subset H$,  we have   $m'\leq n'$.  By Theorem~\ref{theo:reduced form}, there exists $P_1\in H(\k)$ such that $B_{red}=\phi(P_1)BP_{1}^{-1}\in \Gdiag (\k)$ is reduced.  
 Since $H$ is an algebraic group defined over $C$, the equation over $C$ between the entries of $B$ and  $\phi^{k}(B)$, $k\in \Z$, are the same. Then, the class of $\phi^{k}(B)$ in $H/\Gdiag^0$ is independent of $k$. 
 
 Then, the  integer $m'$ (resp. $n'$) is the smallest such that we have the inclusion ${B_{red,[m']}=\phi^{m'-1}(B_{red})\times \dots \times B_{red}
\in \Gdiag^0 (\k)}$ (resp.  $B_{[n']}
\in \Gdiag^0 (\k)$).  Since $\Gdiag^0$ is composed of diagonal matrices and $B,B_{red}$ are not diagonal (otherwise $G_2$ would be reducible),  we find that $m'=2m$, $n'=2n$ for some integers $m,n$. 
We have $B_{red,[2m]}=\phi^{2m}(P_{1})B_{[2m]}P_{1}^{-1} \in \Gdiag^0(\k)$.
 Since $\Gdiag^0 $ is the smallest algebraic group such that  $B_{[2n]}\in \Gdiag^0 (\k)$
  and since $[B_{[2]}]$ is reduced as a $\phi^2$-system,  we have  $n\leq m$,  proving that $m'=n'$.  Then $\Gdiag=H$ and $[B]$ is already reduced. 
 
\medskip
  
$\bullet$  When $G_2$ is irreducible and primitive,  by  \cite[Section 4.4]{hendriks1998algorithm},  either  $G_2=\mathrm{GL}_2 (C)$, or there exists an integer $k$ such that 
$$
G_2 = \lbrace M\in\GL_2\left(\mathbb{C}\right) \mid \det(M)^{k}=1 \rbrace .
$$
We recall that $\det(G_2)$ is the difference Galois group of $\phi(Y)=dY$ with $d:=\det \left(\phi(T_2)A_2 T_2^{-1}\right)$.  
Let $R:=\mathrm{Diag}(r_1,r_2)$ be a reduction matrix of  the bloc diagonal system
 $\phi (Y)=BY$ where $B:=\begin{pmatrix}
\phi(T_1)\alpha T_1^{-1}&0\\ 
0 &d
\end{pmatrix}$.  
Then,  
$$
\phi(R)BR^{-1} = \mathrm{Diag}(\underbrace{\phi(r_1T_1)\alpha (T_1r_1)^{-1}}_{:=B_1}, \underbrace{\phi(r_2)d r_2^{-1}}_{:=B_2})\, .
$$
Let us prove that  $T:=\mathrm{Diag}(r_1,r_2,1)$ is a reduction matrix of $[D]$. Let $ H$ be the minimal algebraic group such that $B_{red}:=\phi(T)D T^{-1}\in H(\k)$. We set $B_{red}=\begin{pmatrix}
B_{red,1} & 0 \\ 0 & B_{red,2}
\end{pmatrix}$, with  $B_{red,\star}$ of convenient size. By construction, we have $B_{red,1} = B_1$ and $\det(B_{red,2}) = B_2$.
Note that $\Gdiag$ is a subgroup of  $G_1\times \det(G_2)  .\SL_2(C)$, where the identification $G_2\sim \det(G_2)  .\SL_2(C)$ is made via the map $\mathscr{M}= \mathrm{Diag}(\det(\mathscr{M}),1)\mathscr{M}'$, with $\det(\mathrm{Diag}(\det(\mathscr{M}),1))=\det(\mathscr{M})$, and $\mathscr{M}'\in \mathrm{SL}_2(C)$.   
Furthermore, since  $G_2$ is irreducible and primitive,  $\Gdiag$ contains $\{1\}\times \SL_2(C)$.  Consider the  map 
$$
\begin{array}{rccl}
\xi :  &   G_1 (\k) \times  \det(G_2). \SL_2 (\k) & \rightarrow & (\k^{*})^2 \\
     & (A,B) & \mapsto & (A,\det(B)) \, .
\end{array}
$$
 Let $ \{1\}\times \SL_2\subset H_1\subsetneq H_2\subset G_1\times G_2$ be two algebraic groups.   Since the kernel of $\xi$ is $ \{1\}\times \SL_2(\k)$,  we have $\xi(H_1(\k))\subsetneq \xi(H_2(\k))$.   Furthermore,  $\xi$ is injective on the set of  algebraic subgroups of  $(G_1\times G_2)(\k)$  containing $ \{1\}\times \SL_2(\k)$. \par 
By  Theorem \ref{theo:reduced form},  we find that $\Gdiag\subset H$ so that we have $\xi(\Gdiag(\k))\subset \xi(H(\k))$.  Since $R:=\mathrm{Diag}(r_1,r_2)$ is  a reduction matrix of  the bloc diagonal system
 $\phi (Y)=BY$, we have 
  $\xi(H(\k))\subset \xi(\Gdiag(\k))$,   and then $\xi(\Gdiag(\k))= \xi(H(\k))$.
From the above injectivity property,   $\Gdiag(\k)=H(\k)$. 
Then $T:=\mathrm{Diag}(r_1,r_2,1)$ is a reduction matrix of $[D]$.
\end{proof} 

  Thus, in all cases, we should be able to reduce the block diagonal system $[D]$ and find an associated reduction matrix, denoted by $T'$.

\bigskip

The next step consists in reducing the system $[A]$.  We will see that finding the reduction matrix is equivalent to finding vector solutions in $\k$ of $\phi$-linear systems with coefficients in $\k$ that we will explicit. We summarize these results in Theorem \ref{thm:reduc_systcomplet}. We consider
$$B:=\phi(T')A T'^{-1}.
$$
The matrix $B$ is of the form $ \begin{pmatrix}
D_1& \widetilde{A} \\ 
0 &D_2
\end{pmatrix}$. By construction, the difference Galois group of $[B]$ is $G$, so let us compute the latter.  We write $B_{\mathcal{D}}:=\begin{pmatrix}
D_1&0 \\ 
0 & D_2
\end{pmatrix}$.  
Recall that  $\Gdiag$ is the Galois group of $[B_{\mathcal{D}}]$,  which is a reduced system thanks to the first step.
The Galois group $G$ is conjugated to a subgroup of 
$$\left\{ \begin{pmatrix}
\mathscr{C}_1& \begin{matrix} c_1& c_2 \end{matrix} \\ 
0 & \mathscr{C}_2
\end{pmatrix}\Bigg|   \begin{pmatrix}
\mathscr{C}_1& 0 \\ 
0 & \mathscr{C}_2
\end{pmatrix} \in \Gdiag,   c_1,c_2\in C \right\} =\Gdiag \ltimes U,
$$
where $U=\left\{ \begin{pmatrix}
1& c_1& c_2 \\ 
0 & 1 &0 \\
0&0&1
\end{pmatrix}\Bigg|   c_1,c_2\in C \right\}$  is isomorphic to two copies of the additive group.  

By Lemma \ref{lem:Gauge transfo reduc}, there exists $T:=\begin{pmatrix}
1& X \\ 
0 & \mathrm{Id}_2
\end{pmatrix}\in \GL_3 (\k)$
such that $[\phi(T)BT^{-1}]$ is reduced. 
 Let  $F:=\phi(T)BT^{-1}$.
 We have $F=  \begin{pmatrix}
D_1& A' \\ 
0 & D_2
\end{pmatrix}$ where $A' = - XD_1 +\widetilde{A}+\phi(X)D_2 $, that is, 
$$\widetilde{A}+\phi(X)D_2=XD_1 +A'.$$
Let $\Gunip=G \cap U$.  It is an algebraic subgroup of $U$ so it is of dimension $0$,  $1$ or $2$.

\begin{lem}
\label{lem:dimGu0}
The dimension of the algebraic group $\Gunip$ is $0$ if and only if there exists $X\in\k^2$ such that 
\begin{equation}
\label{eq:eqverif}
    \widetilde{A}+\phi(X)D_2=XD_1.
\end{equation}
In this case $[F]= [B_{\mathcal{D}}]$ is a reduced system equivalent to $[B]$ with the gauge transformation $T:=\begin{pmatrix}
1& X \\ 
0 & \mathrm{Id}_2
\end{pmatrix}$ and the Galois group $G$ is $\Gdiag$.
\end{lem}

\begin{proof}
From Theorem \ref{theo:reduced form}, the dimension of $\Gunip$ is $0$ if and only if $A' = (0\quad 0) $.
\end{proof}

\begin{lem}
\label{lem:dimGu1}
Assume that the dimension of $\Gunip$ is not $0$. 
 If the dimension of $\Gunip$ is $1$ then $D_2$ is a triangular matrix (lower or upper triangular, including the diagonal matrices).
If it is the case, we can assume that $D_2$ is a lower triangular matrix and three cases occur:

$\bullet$ Case where $D_2$ is lower triangular but not diagonal. In this case 
$$G =  \left\{ \begin{pmatrix}
\mathscr{C}_1& \begin{matrix} c&0 \end{matrix} \\ 
0 & \mathscr{C}_2
\end{pmatrix}\Bigg|   \begin{pmatrix}
\mathscr{C}_1& 0 \\ 
0 & \mathscr{C}_2
\end{pmatrix} \in \Gdiag,   c\in C \right\}.
$$

$\bullet$ Case where $D_2$ is diagonal but not a dilatation.  In this case, $G$ is the same group as before or $\left\{ \begin{pmatrix}
\mathscr{C}_1& \begin{matrix} 0&c \end{matrix} \\ 
0 & \mathscr{C}_2
\end{pmatrix}\Bigg|   \begin{pmatrix}
\mathscr{C}_1& 0 \\ 
0 & \mathscr{C}_2
\end{pmatrix} \in \Gdiag,   c\in C \right\}$, depending if there is a solution in $\k$ of some linear equations with coefficients in $\k$ (see equations \eqref{eqaverifier} or \eqref{eqaverifierbis}). 

$\bullet$ Case where $D_2$ is a dilatation.  In this case,
$$G =  \left\{ \begin{pmatrix}
\mathscr{C}_1& \begin{matrix} \lambda c & \mu c \end{matrix} \\ 
0 & \mathscr{C}_2
\end{pmatrix}\Bigg|   \begin{pmatrix}
\mathscr{C}_1& 0 \\ 
0 & \mathscr{C}_2
\end{pmatrix} \in \Gdiag,   c\in C \right\}.
$$
where $\lambda, \mu\in C^2 \setminus \{(0,0)\}$ can be explicited, see \eqref{eq:eqverif3}.

In each case,  the proof gives an  explicit gauge transformation and a reduced system.
\end{lem}

\begin{proof}
Let $\mathscr{Y}$ be such that $h= \begin{pmatrix}
1&\mathscr{Y} \\ 
0 & \mathrm{Id}_2 
\end{pmatrix} \in \Gunip$.  
For $g=\begin{pmatrix}
a_g & \mathscr{C}_g \\ 
0 & \mathscr{B}_g  
\end{pmatrix} \in G$, let us compute $ghg^{-1}\in \Gunip$.  We write 
$ghg^{-1}= \begin{pmatrix}
1&\mathscr{Z}\\ 
0 & \mathrm{Id}_2 
\end{pmatrix} $,  so that we have $\begin{pmatrix}
a_g & \mathscr{C}_g \\ 
0 & \mathscr{B}_g  
\end{pmatrix}\begin{pmatrix}
1& \mathscr{Y} \\ 
0 & \mathrm{Id}_2 
\end{pmatrix} =\begin{pmatrix}
1&\mathscr{Z} \\ 
0 & \mathrm{Id}_2 
\end{pmatrix} \begin{pmatrix}
a_g & \mathscr{C}_g \\ 
0 & \mathscr{B}_g  
\end{pmatrix}$. 
Then,  $a_g \mathscr{Y} + \mathscr{C}_g = \mathscr{C}_g + \mathscr{Z}\mathscr{B}_g$ proving that
$ghg^{-1}= \begin{pmatrix}
1&a_g \mathscr{Y} \mathscr{B}_g^{-1} \\ 
0 & \mathrm{Id}_2 
\end{pmatrix}  \in \Gunip$.
If the dimension of $\Gunip$ is 1 then  $a_g \mathscr{Y} \mathscr{B}_{g}^{-1}=c\mathscr{Y}$ for some $c\in C^*$. Taking the transposition,  we find $(\mathscr{B}_g^{-1})^{t}\mathscr{Y}^{t}=c'\mathscr{Y}^t$ for some $c'\in C^*$. 
So $\mathscr{Y}^{t}$ is an eigenvector of any elements $(\mathscr{B}_g^{-1})^{t}$.  This means that the elements $\mathscr{B}_g$ are all conjugated to lower triangular matrices (or upper triangular matrices) and $Y^{t}$ is necessarily a common eigenvector. Since $[F]$ is a reduced system, $D_2$ is also a triangular matrix. Thus, if $\Gunip$ has dimension $1$ then $D_2$ is a triangular matrix.\\

Let us  assume that $\Gunip$ has dimension $1$ thus  $D_2$ is a triangular matrix and, up to conjugation, we can assume that $D_2$ is lower triangular. In what follows, we denote by $\widetilde{A}_1$ and $\widetilde{A}_2$ the entries of $\widetilde{A}$. We distinguish three cases:

$\bullet$ Case where $D_2$ is lower triangular but not diagonal. We set $D_2 = \begin{pmatrix}
d_1 & 0 \\ d_3 & d_4
\end{pmatrix}$. In this case, the common eigenvectors of the  elements $(\mathscr{B}_g^{-1})^t$ for $g\in G$ are  $\mathscr{Y}^{t}$ with $\mathscr{Y}:=\begin{pmatrix}
    f & 0 
\end{pmatrix}$ where $f\in C$, thus the system $[B]$ is equivalent to the system $[F]$ where $F$ is of the form $\begin{pmatrix}
D_1& \begin{matrix} f& 0 \end{matrix} \\ 
0 &D_2
\end{pmatrix}$, for a certain $f\in \mathbf{k}^*$, thanks to a gauge transformation $T:=\begin{pmatrix}
1& x_1& x_2 \\ 
0 & 1 &0 \\
0&0&1
\end{pmatrix}$.  Recall that $\widetilde{A}+\phi(X)D_2=XD_1+A'$.  Since the second entry of $A'$ has to be $0$,  we obtain 
\begin{equation}
    \label{eq:eqaverifierD2triang}
    \widetilde{A}_2+\phi(x_2)d_4= D_1 x_2.
\end{equation}
From Theorem \ref{theo:reduced form}, this implies that
$$G\subset \left\{ \begin{pmatrix}
\mathscr{C}_1& \begin{matrix} c&0 \end{matrix} \\ 
0 & \mathscr{C}_2
\end{pmatrix}\Bigg|   \begin{pmatrix}
\mathscr{C}_1& 0 \\ 
0 & \mathscr{C}_2
\end{pmatrix} \in \Gdiag,   c\in C \right\}.
$$
Since $G = \Gdiag\ltimes \Gunip$, we obtain that the previous inclusion is an equality.

$\bullet$ Case where $D_2$ is diagonal but not a dilatation.  We set $D_2 = \mathrm{Diag}(d_1,d_4)$. Up to a conjugation, we have $\Gunip=\left\{ \begin{pmatrix}
1& c&0 \\ 
0 & 1 &0 \\
0&0&1
\end{pmatrix}, c\in C \right\}$ or  $\Gunip=\left\{ \begin{pmatrix}
1& 0&c \\ 
0 & 1 &0 \\
0&0&1
\end{pmatrix}, c\in C \right\}$. 
Using the same notations as before, we obtain that $\Gunip$ is the first one, resp. the second one, if there exists a solution $y\in \k$ of the equation
\begin{align}
    \label{eqaverifier}
    \widetilde{A}_2 + \phi(y)d_4 = D_1y \, ,
\end{align}
resp. of the equation
\begin{align}
    \label{eqaverifierbis}
   \widetilde{A}_1 + \phi(y)d_1 = D_1y  \, .
\end{align}
Not that only one of these two equations has a solution in $\k$ because if it is not the case, the solutions $x_1, x_2$ are such that $X:=(x_1\quad x_2)\in\k^2$ is a solution of \eqref{eq:eqverif}, a contradiction since the dimension of $\Gunip$ is not $0$.
With the same reasoning as before, we obtain that the Galois group $G$ is 
$$\left\{ \begin{pmatrix}
\mathscr{C}_1& \begin{matrix} c & 0 \end{matrix} \\ 
0 & \mathscr{C}_2
\end{pmatrix}\Bigg|   \begin{pmatrix}
\mathscr{C}_1& 0 \\ 
0 & \mathscr{C}_2
\end{pmatrix} \in \Gdiag,   c\in C \right\} \text{ or } \left\{ \begin{pmatrix}
\mathscr{C}_1& \begin{matrix} 0&c \end{matrix} \\ 
0 & \mathscr{C}_2
\end{pmatrix}\Bigg|   \begin{pmatrix}
\mathscr{C}_1& 0 \\ 
0 & \mathscr{C}_2
\end{pmatrix} \in \Gdiag,   c\in C \right\}\, ,
$$
if $\Gunip$ is respectively the first or the second algebraic group.

$\bullet$ Case where $D_2$ is a dilatation. We set $D_2=d\mathrm{Id}_2$ with $d\in \mathbf{k}^*$. Let $\lambda,\mu\in C$ not all zero such that 
$\Gunip=\left\{ \begin{pmatrix}
1& \lambda c& \mu c \\ 
0 & 1 &0 \\
0&0&1
\end{pmatrix}, c\in C \right\}$. From Lemma \ref{lem:Gauge transfo reduc}, since $G = \Gdiag \ltimes \Gunip$, there exists $T:=\begin{pmatrix}
 1 & X \\ 0 & \mathrm{Id}_2
\end{pmatrix}$ such that $\phi(T)B=FT$,
 where $[F]$ is reduced with $F=\begin{pmatrix}
D_1& \begin{matrix} \lambda f& \mu f\end{matrix} \\
0 &D_2
\end{pmatrix}$, $f\in  \mathbf{k}^*$.
This is equivalent to
$$\widetilde{A}+\phi(X)D_2=D_1X+\begin{pmatrix}\lambda f & \mu f  \end{pmatrix}\, ,$$
that is,
$$
\left\lbrace\begin{array}{l}
     d \phi (x_1) = D_1 x_1 - \widetilde{A}_1 +\lambda f  \\
     d \phi (x_2) = D_1 x_2 - \widetilde{A}_2 +\mu f 
\end{array}\right.
$$
where $x_1,x_2$ are the coordinates of $X$. 
Taking $\mu$ times the first line minus $\lambda$ times the second,  we find that $x=\mu x_1-\lambda x_2$ is solution of $d\phi(x)=D_1 x- \mu \widetilde{A}_1 +\lambda  \widetilde{A}_2$. 
We then find that
 this is equivalent to the fact that the vector $Y:= \begin{pmatrix} \lambda\\ \mu \\ x \end{pmatrix}$ is solution of 
\begin{equation}
\label{eq:eqverif3}
    \phi(Y)=
 \begin{pmatrix}1 & 0 & 0  \\
0&1&0\\
\widetilde{A}_2d^{-1} & -\widetilde{A}_1 d^{-1} & D_1d^{-1}
 \end{pmatrix}Y.
\end{equation}
Conversely, any nonzero solution of the latter system in $\mathbf{k}^3$ will give a reduction matrix, and a reduced form.
\end{proof}

Let us summarize the results proved in Lemma \ref{lem:dimGu0} and Lemma \ref{lem:dimGu1}. 

\begin{thm}
\label{thm:reduc_systcomplet}
If the equation \eqref{eq:eqverif} has a solution then the Galois group $G$ is $\Gdiag$. If this equation has no solution then we have the following cases :
\begin{itemize}
    \item if $D_2$ is a triangular matrix but not diagonal, up to a conjugation we assume that it is lower triangular, and if the equation \eqref{eq:eqaverifierD2triang} has a solution then $$G = \left\{ \begin{pmatrix}
\mathscr{C}_1& \begin{matrix} c&0 \end{matrix} \\ 
0 & \mathscr{C}_2
\end{pmatrix}\Bigg|   \begin{pmatrix}
\mathscr{C}_1& 0 \\ 
0 & \mathscr{C}_2
\end{pmatrix} \in \Gdiag,   c\in C \right\}.
$$ 
    \item if $D_2$ is a diagonal matrix but not a dilatation and if \eqref{eqaverifier} (the same equation as \eqref{eq:eqaverifierD2triang})  or \eqref{eqaverifierbis} has a solution then $G$ is the previous group or respectively $$ \left\{ \begin{pmatrix}
\mathscr{C}_1& \begin{matrix} 0&c \end{matrix} \\ 
0 & \mathscr{C}_2
\end{pmatrix}\Bigg|   \begin{pmatrix}
\mathscr{C}_1& 0 \\ 
0 & \mathscr{C}_2
\end{pmatrix} \in \Gdiag,   c\in C \right\}.
$$ 
    \item if $D_2$ is a dilatation and \eqref{eq:eqverif3} has a solution then $$G = \left\{ \begin{pmatrix}
\mathscr{C}_1& \begin{matrix} \lambda c&\mu c\end{matrix} \\ 
0 & \mathscr{C}_2
\end{pmatrix}\Bigg|   \begin{pmatrix}
\mathscr{C}_1& 0 \\ 
0 & \mathscr{C}_2
\end{pmatrix} \in \Gdiag,   c\in C \right\} 
$$  where $\lambda, \mu\in C$ not all zero are given by a solution of \eqref{eq:eqverif3}. 
    \item if none of these three cases occur then $G=\Gdiag\ltimes U$ (in this case, the dimension of $\Gunip$ is $2$).
\end{itemize}
\end{thm}

\subsection{Imprimitive case}
Assume now that $G$ is irreducible and treat the case where $G$ is imprimitive. This case is  solved in Section \ref{sec:imprimitive case}.
 It is imprimitive if and only if the difference Galois group of $\phi^3 Y=A_{[3]}Y$ is diagonal. We use the reducible case (with $\phi$ replaced by $\phi^3$) to check this property.   In that case we are able to compute $d\in \k^*$,  solution of the Riccati equation corresponding to $\phi^3 Y=A_{[3]}Y$.  We now follow Proposition \ref{prop2}. 
Let us   compute a reduced form of $\phi(y)=dy$, that is compute $f\in \k^*$ such that $\phi(y)=d\frac{\phi(f)}{f}y$ is reduced.  Then, a reduced form is given by 
$\left( \begin{array}{ccc}
0&1&0\\
0 &0&  f^{-1}\\
\phi(f)d & 0&0
\end{array}\right) $ with reduction matrix $\mathrm{Diag}(1,1,f)$.

\subsection{Primitive case}

Assume now that $G$ is irreducible and primitive.  Let $D(G^0)$ be the derived subgroup of $G^0$ that is irreducible and primitive by Proposition \ref{prop4}. By definition, $D(G^0)\subset \mathrm{SL}_{3}(C)$.   By \cite[Section 2.2]{singer1993galois}, either $D(G^0)$ is $\mathrm{SL}_3(C)$, or it is isomorphic to $\mathrm{PSL}_2(C)$.  Since $\mathrm{SO}_3(C)$ is also a primitive subgroup of  $\mathrm{SL}_3(C)$,  it is isomorphic to $\mathrm{PSL}_2(C)$.  
The matrices commuting with   $\mathrm{SL}_3(C)$  (resp. $\mathrm{SO}_3(C)$)  are dilatations and we find that $Z(G^0)$, the center of $G^0$, is isomorphic to an algebraic subgroup of  $C^*$.  Since it is connected, it is then $C^*$ or $\{1\}$. By Proposition~\ref{prop4}, $G^0=Z(G^0).D(G^0)$.  \par 
When $D(G^0)=\mathrm{SL}_3(C)$ we find that either $G^0=\mathrm{GL}_3(C)$, or 
$G^0=\mathrm{SL}_3(C)$. Since $G/G^0$ is cyclic, we deduce that  either $G=\mathrm{GL}_3(C)$, or there exists an integer $k$, which is the order of the group $G/G^0$, such that 
$G=\lbrace M\in\GL_3\left( C\right) \mid \det(M)^{k}=1 \rbrace $. Then $G=Z(G).D(G)$, where $Z(G)$ is composed by dilatations matrices and $D(G)=\mathrm{SL}_3(C)$.\par 
Assume now that $D(G^0)=\mathrm{SO}_3(C)$. Then  either $G^0 \sim C^* .\mathrm{SO}_3(C)$, or $G^0=\mathrm{SO}_3(C)$. Moreover, since $G/G^0$ is cyclic, it is abelian and since the derived group $D(G)$ is the smallest subgroup $H$ of $G$ such that $G/H$ is abelian, we have $D(G)\subset G^0$.
Then $\mathrm{SO}_3(C)=D(G^0)\subset D(G)\subset G^0 \subset C^*. \mathrm{SO}_3(C)$. The normalizer of $\mathrm{SO}_3(C)$ in $\mathrm{GL}_3(C)$ is $C^* . \mathrm{SO}_3(C)$ and the same holds for $C^*. \mathrm{SO}_3(C)$. Since the normalizer of $D(G)$ in $G$, which is $G$, is included in the normalizer of $D(G)$ in $\GL_3(C)$, we find that $G\subset C^*. \mathrm{SO}_3(C)$. Thus, $\mathrm{SO}_3(C)\subset G\subset C^*. \mathrm{SO}_3(C)$.
 Then, either $G\sim C^*. \mathrm{SO}_3(C)$, or there exists $k\in \N^*$, such that $G\sim U_k .\mathrm{SO}_3(C)$, where $U_k$ denote the group of all k-th roots of unity. Hence $G=Z(G).D(G)$, where $Z(G)$ is composed by dilatations matrices and $D(G)=\mathrm{SO}_3(C)$.
 To summary, we have proved the following result:
 \begin{prop}\label{prop5}
 Assume that $G$ is irreducible and primitive. Then, $G=Z(G).D(G)$, where $Z(G)$ is included in the group of scalar matrices and $D(G)$ is either $\mathrm{SL}_3(C)$ or $\mathrm{SO}_3(C)$. 
 \end{prop}
\begin{rem}
If $D(G)=\mathrm{SL}_3(C)$, then
$ 
Z(G) = \lbrace\lambda \mathrm{Id}_3  \text{ where } \lambda^3\in \det(G) \rbrace\, .
$
Recall that  $\det(G)$ is the Galois group of an equation of order one $\phi (y)=\det (A)y$. 

Assume now $D(G)=\mathrm{SO}_3(C)$. If $\det(G)=C^*$, then $Z(G) = C^*$. If $\det(G)=U_n$ with $\gcd(n,3)\neq 1$, then $Z(G)=U_{3n}$. In the case $\det(G)=U_n$ with $\gcd(n,3)= 1$, we have $Z(G)=U_{3n}$ or $U_n$.
\end{rem}

We want to give a criterion to decide whether $D(G)=\mathrm{SO}_3(C)$ or not.  Let us prove the following lemma. 

 \begin{lem}\label{lem3}
Assume that $G$ is irreducible and primitive. Let  $\overline{\mathbf{k}}$ be the algebraic closure of $\k$. We may extend $\phi$ on the latter.   The following facts are equivalent. 
\begin{itemize}
\item  $D(G)=\mathrm{SO}_3(C)$.
\item there exist $T\in \mathrm{GL}_3(\k)$,  $\lambda \in \mathbf{k}^*$  and $B\in  \mathrm{SO}_3(\mathbf{k})$ such that   \begin{equation}\label{eq1}
 \phi(T)AT^{-1}=\lambda B.
 \end{equation}
\item there exist $T\in \mathrm{GL}_3(\overline{\mathbf{k}})$,  $\lambda \in \overline{\mathbf{k}}^*$  and $B\in  \mathrm{SO}_3(\overline{\mathbf{k}})$ such that  \eqref{eq1} holds. 
\end{itemize}
\end{lem}

\begin{proof}
Recall that   $G\sim Z(G). D(G)$,  where $D(G)\in \{\mathrm{SO}_3(C),\mathrm{SL}_3(C) \}$ and $Z(G)$ is composed by dilatations matrices.  By Theorem~\ref{theo:reduced form},  $D(G)=\mathrm{SO}_3(C)$ if and only if 
there exist $T\in \mathrm{GL}_3(\k)$,  $\lambda \in \mathbf{k}^*$   and $B\in  \mathrm{SO}_3(\mathbf{k})$ such that  \eqref{eq1} holds.  So the first two items are equivalent.   If the second item holds, then the third item holds.  Conversely assume that the third item holds.  
Consider $\overline{G}$,  the Galois group over $\overline{\mathbf{k}}$.  By Theorem \ref{theo:reduced form},  $\overline{G}\subset \overline{C}^*. \mathrm{SO}_3(\overline{C})$. 
 By Galois correspondence, see \cite[Chapter 1]{van2006galois},  $\overline{G}$ is a subgroup of $G$ and $G/\overline{G}$ is  finite. Since $G\sim  Z(G). D(G)$,  where $D(G)\in \{\mathrm{SO}_3(C),\mathrm{SL}_3(C) \}$, this shows that $D(G)$ cannot be $\mathrm{SL}_3(C)$.  Then $D(G)= \mathrm{SO}_3(C)$. 
\end{proof}

The equation \eqref{eq1} is non linear and may be complicated to solve. Let us see how we may reduce to a linear equation over $\overline{\mathbf{k}}$.  
For the sake of completeness, let us first prove a well-known
result that will  be used in the sequel.
 
 \begin{lem}\label{lem4}
 Let $K$ be a field of characteristic zero.   
Let $X\in \mathrm{GL}_n (K)$ be a symmetric matrix. Then, there exist $F,D\in \mathrm{GL}_n (K)$ with $D$ diagonal such that $X=FDF^t$. 
 \end{lem}
 
 \begin{proof}
 Let $E_{i,j}$ be the square matrix of order $n$ whose $(i,j)$ entry is $1$ and the other entries are $0$. 
For $i\neq j$, let $T_{i,j}$ be the corresponding  transvection matrix, that is $\mathrm{Id}_n+E_{i,j}$. 
If we consider 
$T_{i,j}X(T_{i,j})^{t}$ we make the following operations: the column $C_i$ of $X$ is replaced by $C_{j}+C_i$ and then the line $L_i$ is replaced by $L_{i}+L_j$. Take $i=1$. This shows that $X_{1,1}$ is replaced by $X_{1,1}+X_{j,1}+X_{1,j} +X_{j,j}=X_{1,1}+2X_{1,j}+X_{j,j}$. \par 
Let us see how to reduce to the case $X_{1,1}\neq 0$.  Assume $X_{1,1}=0$. If there exists $j>1$ such that $X_{j,j}\neq 0$ then we replace $X$ by $PXP^t$ where $P = E_{1,j} + E_{j,1} + \sum\limits_{k\neq 1,j}E_{k,k}$ is the permutation matrix which exchanges the lines $1$ and $j$. We have $P= P^{-1} = P^{t}$. If $X_{j,j}=0$ for all $j$, we consider $j$ such that $X_{1,j}\neq 0$ (such $j$ exists otherwise the first line is only composed of $0$ entries, contradicting  the fact that $X$ is invertible).  Let us perform 
$T_{1,j}X(T_{1,j})^{t}$ so that $X_{1,1}$ is replaced by $X_{1,1}+2X_{1,j}+X_{j,j}=2X_{1,j}\neq 0$.   So we may reduce to the case where $X_{1,1}\neq 0$. \par 
For $i\in \lbrace 2,\ldots,n\rbrace$, let $T  := \mathrm{Id}_n + \alpha_i E_{i,1}$ with $\alpha_i:=- X_{i,1}/X_{1,1}$. Computing $TXT^{t}$, $X_{i,1}$ is replaced by  $\alpha_i X_{1,1} + X_{i,1} = 0$ thus, as in a Gaussian Pivot, we force $X_{2,1}=\dots =X_{n,1}=0$. We also have  $X_{1,2}=\dots =X_{1,n}=0$ since the matrix obtained remains symmetric.  We repeat the procedure to  the other lines and columns to find the result. 
 \end{proof}
 
 Toward an efficient way to check that \eqref{eq1} holds,  let us prove the following. 
 
 \begin{lem}\label{lem6}
The following facts are equivalent. 
\begin{itemize}
\item There exist $T\in \mathrm{GL}_3(\overline{\mathbf{k}})$,  $\lambda \in \overline{\mathbf{k}}^*$  and $B\in  \mathrm{SO}_3(\overline{\mathbf{k}})$ such that \eqref{eq1} holds. 
\item There exist $\lambda \in \overline{\mathbf{k}}^* $,  $X\in \mathrm{GL}_3(\overline{\mathbf{k}})$ such that $X=X^t$,  and  $A XA^t =\lambda^2\phi(X)$.
\end{itemize}
\end{lem}

\begin{proof}
Let us prove that the first point implies the second.  
Let us assume the existence of  $T\in \mathrm{GL}_3(\overline{\mathbf{k}})$,  $\lambda \in \mathbf{\overline{\mathbf{k}}}^*$,  and $B\in  \mathrm{SO}_3(\overline{\k})$ such that \eqref{eq1} holds. Let us set $P:=T^{-1}$ and note that $P^t$ is the inverse of $T^t$.  Using $BB^t=\mathrm{Id}_3$ we find that \eqref{eq1} implies
 $$\phi(T)APP^t A^t \phi(T^t)=\lambda^2 \mathrm{Id}_3.$$
The latter is equivalent to 
 $$APP^t A^t =\lambda^2\phi(PP^t).$$
 We obtain the result if we set $X=PP^t$. \par 
 Conversely,  assume that the second item holds.  By Lemma \ref{lem4},  let us write $X=FDF^t$ with matrices whose coefficients are in $\overline{\mathbf{k}}$. Since $\overline{\mathbf{k}}$ is algebraically closed,  $F\sqrt{D}$ has coefficients in $\overline{\mathbf{k}}$ and we may write 
 $X=F\sqrt{D}(\sqrt{D})^t F^t$, where $\sqrt{D}$ is any diagonal matrix whose square is $D$.  Let us perform $\phi(T)AT^{-1}$ where $T\in \mathrm{GL}_3(\overline{\mathbf{k}})$ is the inverse of $F\sqrt{D}$.  Note that $X=T^{-1}(T^{-1})^t $.  It suffices to prove that $B=\phi(T)AT^{-1}\lambda^{-1}\in \mathrm{SO}_3(\overline{\mathbf{k}})$. 
 Then, 
 $$B B^t=\lambda^{-2} \phi(T)AT^{-1}(T^{-1})^t A^{t}\phi(T^t)=\lambda^{-2} \phi(T)AX A^{t}\phi(T^t).$$
 We now use $A XA^t =\lambda^{2}\phi(X)$ to deduce that the latter equals 
 $ \phi(T)\phi(X) \phi(T^t). $
With $X=T^{-1}(T^{-1})^t $, we deduce that  $\phi(T)\phi(X) \phi(T^t)$ is the identity, that is $B\in  \mathrm{SO}_3(\overline{\mathbf{k}})$. 
\end{proof}

If we replace $X$ by $X':=X\det(X)^{-1/3}\in \mathrm{SL}_3 (\overline{\mathbf{k}})$,  we find $A X'A^t =\det(A)^{2/3}\phi(X')$. Indeed, the equality $AXA^t=\lambda^2 \phi(X)$ is equivalent to 
\begin{equation}
\label{eq:for_determinant}
AX'A^t = \lambda^2\det(X)^{-1/3}\det(\phi(X))^{1/3} \phi(X')\, .
\end{equation}
Taking the determinant in \eqref{eq:for_determinant} and using that $\det(X')=1$, we have $\det(A)^2=\lambda^6\det(X)^{-1}\det(\phi(X))$. We obtain the desired result thanks to 
 \eqref{eq:for_determinant} and this last equality.
With Lemma \ref{lem3},  and Lemma \ref{lem4},  we have proved the following result.
 
\begin{prop}\label{prop3}
Assume that $G$ is irreducible and primitive.   Then,  $D(G)=\mathrm{SO}_3(C)$ if and only if there exists $X\in \mathrm{SL}_3(\overline{\mathbf{k}})$ such that $X=X^t$ and  $A XA^t =\det(A)^{2/3}\phi(X)$.
\end{prop}

Proposition \ref{prop3} then reduces the criterion on $D(G)$ to look at solutions in $\overline{\mathbf{k}} $ of linear $\phi$-difference systems with coefficients in $\overline{\mathbf{k}}$. More precisely the condition $X=X^t$ drops the number of unknowns from $9$ to 
$6$.  
Solving a system in $\overline{\mathbf{k}}$ is in general a complicated task but in many examples, by studying the poles of the system for instance, we might be able to prove that the system has no solutions, showing that $D(G)\neq \mathrm{SO}_3(C)$.

Let us now see  how to compute the reduced form.  
First, assume that $D(G)=\mathrm{SL}_3 (C)$. 
We compute $\det(G)$ as the difference Galois group of an order one system $\phi(y)=a_0 y$ with reduction matrix $f\in \k^*$, that is $\phi(y)=a_0\frac{\phi(f)}{f} y$ is reduced with Galois group $\det(G)$.
Then, we perform, as in \cite[Section 4.4]{hendriks1998algorithm}, the gauge transformation  $T=\mathrm{Diag}(f,1,1)$ to obtain a system $[B]$.  Since $\det(B)=a_0\frac{\phi(f)}{f}$,   and $\phi(y)=a_0\frac{\phi(f)}{f}y$ is reduced with Galois group $\det(G)$,  we find that $B\in G(\mathbf{k})$.  Then,  $B$ is a reduced form and $T$ is the reduction matrix.   \par
Now consider the case where $D(G)= \mathrm{SO}_3(C)$. 
By Lemma \ref{lem3}, we know that there exist $T\in \mathrm{GL}_3(\k)$,  $\lambda \in \mathbf{k}^*$  and $B\in  \mathrm{SO}_3(\mathbf{k})$ such that \eqref{eq1} holds. As above, let us set $X'=T^{-1}(T^{-1})^{t}\det(T)^{2/3}\in \mathrm{SL}_3 (\overline{\mathbf{k}})$. Then, $X'$ satisfies $X'=(X')^t$ and  $A X'A^t =\det(A)^{2/3}\phi(X')$. Let us write $X'=FDF^t$ as in Lemma \ref{lem4}.   Note that the proof of Lemma \ref{lem4} is constructive so $F$ and $D$ may be effectively computed. 
Let us perform the gauge transformation  $\phi(\alpha (F\sqrt{D})^{-1})A(\alpha (F\sqrt{D})^{-1})^{-1}$ where $(F\sqrt{D})^{-1}\in \mathrm{GL}_3(\overline{\mathbf{k}})$ and $\alpha \in \overline{\mathbf{k}}^*$. From what precedes, all reduction matrices with coefficients in $\overline{\mathbf{k}}$  are of this form. Since there is a reduction matrix with coefficients in $\k$,
there is a choice of $X'$  and $\alpha$ such that the corresponding $\alpha (F\sqrt{D})^{-1}$ has entries in $\k$. Note that this procedure is purely theoretical and not effective. So if $\alpha (F\sqrt{D})^{-1}$ does not have entries in  $\mathbf{k}$ for all $\alpha \in \overline{\mathbf{k}}^*$,  we take another solutions $X'$ 
with a corresponding $\alpha (F\sqrt{D})^{-1}\in \mathrm{GL}_3(\mathbf{k})$. At this stage, we have $\phi(\alpha (F\sqrt{D})^{-1})A(\alpha  (F\sqrt{D})^{-1})^{-1}=\beta A_0$, with $\beta \in \mathbf{k}^*$  and $A_0\in \mathrm{SO}_3 (\mathbf{k})$. By Theorem~\ref{theo:reduced form} and Proposition \ref{prop5}, we find that a reduction matrix is of the form $\alpha_1 \alpha (F\sqrt{D})^{-1}$, $\alpha_1 \in \mathbf{k}^*$.  We may compute $\alpha_1$ by reducing the system $\phi (y)=\beta y$. Hence the reduced form is $\frac{\phi (\alpha_1)}{\alpha_1}\beta A_0$.

\begin{ex}\label{ex4}
Consider $G$ the difference Galois group of \eqref{eq0} over $\C(z^{1/*})$. Assume that $t\neq 0$. Assume by contradiction that $D(G)=\mathrm{SO}_3(\C)$.
The eigenvalues of matrices of $\mathrm{SO}_3(\C)$ are of the form $(1,a^{\pm 1})$,  $a\in \C^*$.  Then,  the eigenvalues of elements of $G$ are for the form $(c,ca^{\pm 1})$,  $c,a\in \C^*$.   By \cite[Section 2.4]{RS2009}, combined with the computation of the Newton polygon in Example \ref{ex2}, $G$ contains the so-called theta torus,  elements of the form $\mathrm{Diag}(1,1,d)$, $d\in \C^*$.  If the triple  $\mathrm{Diag}(1,1,d)$ is of the form $(c,ca^{\pm 1})$, then either $a=1=c=d$, or $a=c=-1$ and $d=-1$.  Then,  $\mathrm{Diag}(1,1,2)$ can never be of the form $(c,ca^{\pm 1})$,  showing that $D(G)\neq \mathrm{SO}_3(\C)$.  Then $D(G)=\mathrm{SL}_3(\C)$. 
\end{ex}

 \subsection{Summary}
  Let 
 $$
 L := \phi^3 + a_2 \phi^2 +a_1 \phi +a_0\in \k\left< \phi,\phi^{-1}\right>
 $$  and let $G\subset \GL_{3}(C)$ be the Galois group of $Ly=0$.
 
If the Riccati  equation  attached to $L$ or the similar equation attached to $L^{\vee}$  has a solution then $G$ is reducible. Otherwise, it is irreducible.
\medskip

When it is the reducible case, $L$ or  $L^{\vee}$ has a right factor of order $1$ denoted by $\phi-\alpha$. Consider the case where $L$  has a right factor of order $1$, the other situation is similar. Therefore, up to an isomorphism, the Galois group $G$ is the one of the system 
$$\phi (Y)= \left(\begin{array}{c|c}
\alpha &  1 \hspace*{0.3cm} 0  \\
\hline
0 &  A_2 
\end{array}\right)Y \, . $$
We already know how to reduce the systems $[\alpha]$ and $[A_2]$ (order $1$ and $2$) and we denote respectively by $[\beta]$ and $[B_2]$ these reduced systems. We consider the system $[A]$ which is equivalent to this previous system with $A$ of the form  
$\begin{pmatrix}
\beta &\star \\ 
0 & B_2
\end{pmatrix}$. To reduce the system $[A]$ and compute its Galois group, which is $G$ up to an isomorphism, we proceed in two steps:
\begin{enumerate}
    \item With Theorem \ref{thm:reduc_blocdiag}, we reduce the bloc diagonal system $[D]$ where $D = \begin{pmatrix}
\beta &0 \\ 
0 & B_2
\end{pmatrix}$ and we find an associated reduction  matrix denoted by $T'$.
\item We reduce the system $[B]$ where $B:=\phi(T')A T'^{-1}$. Doing this, we introduce some linear equations and we know which is the group $G$ if we know that there exists a solution or not of these equations (see equations \eqref{eq:eqverif}, \eqref{eqaverifier}, \eqref{eqaverifierbis}, \eqref{eq:eqverif3}) 
\end{enumerate}

\medskip

When it is not the reducible case, we have to distinguish between the imprimitive case and the primitive case. Let $A:=A_L$. We know that it is imprimitive if and only if the difference Galois group of $\phi^3 Y=A_{[3]}Y$ is diagonal. We use the reducible case (with $\phi$ replaced by $\phi^3$) to check this property.  

\medskip

Finally,  when it is irreducible and primitive, we have to distinguish between $D(G)=\mathrm{SO}_3 (C)$ and  $D(G)=\mathrm{SL}_3 (C)$. We have $D(G)=\mathrm{SO}_3 (C)$, if and only if there is a symmetric solution over $\overline{\mathbf{k}}^*$ of the linear difference equation $A XA^t =\det(A)^{2/3}\phi(X)$.  

Combining the facts proved in Examples \ref{ex1},  \ref{ex2},  \ref{ex3}, and  \ref{ex4}, we then find the Galois group of \eqref{eq0} over  $\C(z^{1/*})$:

\begin{thm}\label{thm2}
For all $t\neq 0$, the difference Galois group of \eqref{eq0} over  $\C(z^{1/*})$ is $\mathrm{GL}_3(\C)$.
\end{thm}

\section{Application to the differential transcendence}\label{sec:appliDiffalTranscendence}

Let $\k$ be a field of characteristic zero and $(\k,\phi,\partial)$ be a $\mathcal{C}^1$-field with no proper finite difference field extension,   equipped with an automorphism $\phi$, and a derivation $\partial$, that is an additive morphism satisfying the Leibnitz rule $\partial(fg)=\partial(f)g+f\partial(g)$. Let us assume that $\phi$ and $\partial$ commute.  In our examples, we may take $\partial=\partial_z$ (case S and E),  and $\partial=z \partial_z$ (case Q).  In the Mahler case,  we could adapt as in \cite{dreyfus2018hypertranscendence} the framework with the derivation $\partial= z \partial_z \log(z)$ but this will complicate the exposition of a general criterion, so let us drop the Mahler case in what follows. \par 
Let us consider $R$, a ring extension of $\k$ such that $\phi$ and $\partial$ extend on $R$.  We say that $f\in R$ is $\partial$-algebraic over $\k$ if there exists $n\in\N$ and a nonzero $P\in \k [X_0,\dots, X_n]$ such that $P(f,\dots, \partial^n f)=0$.  We say that $f$ is $\partial$-transcendent over $\k$ otherwise.  In  our examples,  $f$ is $\partial$-algebraic over $\k$ if and only if it is $\partial_z$-algebraic over $\k$.  Many proof of $\partial$-transcendence of solutions of $\phi$-equation are based on the difference Galois theory.  Roughly speaking, this kind of theorem says that if the difference Galois group is sufficiently big, then the nonzero solutions of the linear equation are $\partial$-transcendent.  Let $f\in R^*$ be a solution of 
$$ a_3\phi^3 (y)+ a_2 \phi^2 (y)+a_1 \phi (y)+a_0 y=0, \quad \text{where}\quad  a_i\in \k\quad \text{and}\quad  a_0 a_3 \neq 0.$$
Consider the dual 
$\sum_{i=0}^{3}  \phi^{-i} (a_{i})
\phi^{-i}$ that is equivalent to 
$$ \phi^3(a_0)\phi^3 (y)+ \phi^2(a_1) \phi^2 (y)+\phi(a_2) \phi (y)+a_3 y=0, \quad \text{where}\quad  a_i\in \k.$$

\begin{thm}

Let us assume that 
\begin{itemize}
\item there is no $u\in \k$ such that
$$a_3 u\phi(u)\phi^2 (u)+a_2 u\phi(u)+ a_1 u+a_0=0.$$
\item there is no $u\in \k$ such that
$$\phi^3(a_0) u\phi(u)\phi^2 (u)+\phi^2(a_1) u\phi(u)+ \phi(a_2) u+a_3=0.$$
\item there is no $g\in \k$ and nonzero linear differential operator $L\in C[\partial]$, such that  
$$L(\partial (a_0/a_3 )/(a_0/a_3))=\phi(g)-g. $$
\end{itemize}
Then,  $f$ is $\partial$-transcendent over $\k$.
\end{thm}

\begin{proof}
Let $G$ be the difference Galois group of the equation of order three satisfied by $f$.  The first assumption ensures that the operator has no right factor of order one and the second assumption ensures that the operator has no left factor of order one.  So the operator is irreducible and then the difference Galois group is irreducible.  When $G$ is  imprimitive, this is the same reasoning as \cite[Theorem 3.5]{arreche2021differential}.  When $G$ is primitive, it contains $\mathrm{SO}_3 (C)$ and   we now follow the proof  of \cite[Theorem 3.5]{arreche2021differential}. 
\end{proof}

\bibliographystyle{alpha}
\bibliography{biblio}

\newcommand{\etalchar}[1]{$^{#1}$}
\begin{thebibliography}{NvdPT{\etalchar{+}}08}

\bibitem[AB98]{AB98}
S.~A. Abramov and M.~A. Barkatou.
\newblock Rational solutions of first order linear difference systems.
\newblock In {\em Proceedings of the 1998 {I}nternational {S}ymposium on
  {S}ymbolic and {A}lgebraic {C}omputation ({R}ostock)}, pages 124--131. ACM,
  New York, 1998.

\bibitem[AB01]{AB2001}
Sergei~A. Abramov and Manuel Bronstein.
\newblock On solutions of linear functional systems.
\newblock In {\em Proceedings of the 2001 {I}nternational {S}ymposium on
  {S}ymbolic and {A}lgebraic {C}omputation}, pages 1--6. ACM, New York, 2001.

\bibitem[ABPS21]{abramov2021rational}
Sergei~A Abramov, Manuel Bronstein, Marko Petkov{\v{s}}ek, and Carsten
  Schneider.
\newblock On rational and hypergeometric solutions of linear ordinary
  difference equations in {$\Pi$$\Sigma$}*-field extensions.
\newblock {\em Journal of Symbolic Computation}, 107:23--66, 2021.

\bibitem[Abr95]{Abr95}
S.~A. Abramov.
\newblock Rational solutions of linear difference and {$q$}-difference
  equations with polynomial coefficients.
\newblock {\em Program. Comput. Software}, 21(6):273--278, 1995.

\bibitem[Abr02]{Abr02}
S.~A. Abramov.
\newblock A direct algorithm to compute rational solutions of first order
  linear {$q$}-difference systems.
\newblock {\em Discrete Math.}, 246:3--12, 2002.
\newblock Formal power series and algebraic combinatorics (Barcelona, 1999).

\bibitem[ADH21]{adamczewski2021hypertranscendence}
Boris Adamczewski, Thomas Dreyfus, and Charlotte Hardouin.
\newblock Hypertranscendence and linear difference equations.
\newblock {\em Journal of the American Mathematical Society}, 34(2):475--503,
  2021.

\bibitem[ADHW20]{adamczewski2020algebraic}
Boris Adamczewski, Thomas Dreyfus, Charlotte Hardouin, and Michael Wibmer.
\newblock Algebraic independence and linear difference equations.
\newblock {\em arXiv preprint arXiv:2010.09266}, 2020.

\bibitem[ADR21]{arreche2021differential}
Carlos~E Arreche, Thomas Dreyfus, and Julien Roques.
\newblock Differential transcendence criteria for second-order linear
  difference equations and elliptic hypergeometric functions.
\newblock {\em Journal de l’{\'E}cole polytechnique—Math\'{e}matiques},
  8:147--168, 2021.

\bibitem[AMP22]{amzallag2022degree}
Eli Amzallag, Andrei Minchenko, and Gleb Pogudin.
\newblock Degree bound for toric envelope of a linear algebraic group.
\newblock {\em Mathematics of Computation}, 91(335):1501--1519, 2022.

\bibitem[APP98]{APP98}
Sergei~A. Abramov, Peter Paule, and Marko Petkov\v{s}ek.
\newblock {$q$}-hypergeometric solutions of {$q$}-difference equations.
\newblock In {\em Proceedings of the 7th {C}onference on {F}ormal {P}ower
  {S}eries and {A}lgebraic {C}ombinatorics ({N}oisy-le-{G}rand, 1995)}, volume
  180, pages 3--22, 1998.

\bibitem[AS17]{arreche2017galois}
Carlos~E Arreche and Michael~F Singer.
\newblock Galois groups for integrable and projectively integrable linear
  difference equations.
\newblock {\em Journal of Algebra}, 480:423--449, 2017.

\bibitem[BCWDV16]{BCWV}
Moulay Barkatou, Thomas Cluzeau, Jacques-Arthur Weil, and Lucia Di~Vizio.
\newblock Computing the {L}ie algebra of the differential {G}alois group of a
  linear differential system.
\newblock In {\em Proceedings of the 2016 {ACM} {I}nternational {S}ymposium on
  {S}ymbolic and {A}lgebraic {C}omputation}, pages 63--70, 2016.

\bibitem[CDDM18]{CDDM18}
Fr{\'e}d{\'e}ric Chyzak, Thomas Dreyfus, Philippe Dumas, and Marc Mezzarobba.
\newblock Computing solutions of linear {M}ahler equations.
\newblock {\em Math. Comp.}, 87:2977--3021, 2018.

\bibitem[Com22]{combot2022hyperexponential}
Thierry Combot.
\newblock Hyperexponential solutions of elliptic difference equations.
\newblock {\em arXiv preprint arXiv:2205.00041}, 2022.

\bibitem[CS99]{CS99}
Elie Compoint and Michael~F. Singer.
\newblock Computing galois groups of completely reducible differential
  equations.
\newblock {\em Journal of Symbolic Computation}, 28(4):473--494, 1999.

\bibitem[CvH06]{CvH06}
Thomas Cluzeau and Mark van Hoeij.
\newblock Computing hypergeometric solutions of linear recurrence equations.
\newblock {\em Appl. Algebra Engrg. Comm. Comput.}, 17(2):83--115, 2006.

\bibitem[DH21]{DH21}
Thomas Dreyfus and Charlotte Hardouin.
\newblock Length derivative of the generating function of walks confined in the
  quarter plane.
\newblock {\em Confluentes Math.}, 13(2):39--92, 2021.

\bibitem[DHR18]{dreyfus2018hypertranscendence}
Thomas Dreyfus, Charlotte Hardouin, and Julien Roques.
\newblock Hypertranscendence of solutions of mahler equations.
\newblock {\em Journal of the European Mathematical Society}, 20(9):2209--2238,
  2018.

\bibitem[DHR21]{dreyfus2021functional}
Thomas Dreyfus, Charlotte Hardouin, and Julien Roques.
\newblock Functional relations for solutions of q-difference equations.
\newblock {\em Mathematische Zeitschrift}, pages 1--41, 2021.

\bibitem[DHRS18]{dreyfus2018nature}
Thomas Dreyfus, Charlotte Hardouin, Julien Roques, and Michael~F Singer.
\newblock On the nature of the generating series of walks in the quarter plane.
\newblock {\em Inventiones mathematicae}, 213(1):139--203, 2018.

\bibitem[DR15]{dreyfus2015galois}
Thomas Dreyfus and Julien Roques.
\newblock Galois groups of difference equations of order two on elliptic
  curves.
\newblock {\em SIGMA. Symmetry, Integrability and Geometry: Methods and
  Applications}, 11:003, 2015.

\bibitem[DW19]{dreyfus2019computing}
Thomas Dreyfus and Jacques-Arthur Weil.
\newblock Computing the lie algebra of the differential galois group: the
  reducible case.
\newblock {\em arXiv preprint arXiv:1904.07925}, 2019.

\bibitem[Fen15]{Feng15}
Ruyong Feng.
\newblock Hrushovski's algorithm for computing the galois group of a linear
  differential equation.
\newblock {\em Advances in Applied Mathematics}, 65:1--37, 2015.

\bibitem[Fen18]{Feng18}
Ruyong Feng.
\newblock On the computation of the {G}alois group of linear difference
  equations.
\newblock {\em Math. Comp.}, 87(310):941--965, 2018.

\bibitem[Hen97]{hendriks1997algorithm}
Peter~A Hendriks.
\newblock An algorithm for computing a standard form for second-order linear
  q-difference equations.
\newblock {\em Journal of pure and applied algebra}, 117:331--352, 1997.

\bibitem[Hen98]{hendriks1998algorithm}
Peter~A. Hendriks.
\newblock An algorithm determining the difference galois group of second order
  linear difference equations.
\newblock {\em Journal of Symbolic Computation}, 26(4):445--462, 1998.

\bibitem[Hru02]{Hru2002}
Ehud Hrushovski.
\newblock Computing the galois group of a linear differential equation.
\newblock {\em Banach Center Publications}, 58(1):97--138, 2002.

\bibitem[HS08]{HS08}
Charlotte Hardouin and Michael~F. Singer.
\newblock Differential {G}alois theory of linear difference equations.
\newblock {\em Math. Ann.}, 342(2):333--377, 2008.

\bibitem[HSS16]{hardouin2016galois}
Charlotte Hardouin, Jacques Sauloy, and Michael~F Singer.
\newblock {\em Galois theories of linear difference equations: an
  introduction}, volume 211.
\newblock American Mathematical Soc., 2016.

\bibitem[Kov86]{kovacic1986algorithm}
Jerald~J Kovacic.
\newblock An algorithm for solving second order linear homogeneous differential
  equations.
\newblock {\em Journal of Symbolic Computation}, 2(1):3--43, 1986.

\bibitem[NvdPT{\etalchar{+}}08]{nguyen2008algebraic}
KA~Nguyen, Marius van~der Put, Jakob Top, et~al.
\newblock Algebraic subgroups of $\mathrm{GL}_{2}(\mathbb{C})$.
\newblock {\em Indagationes Mathematicae}, 19(2):287--297, 2008.

\bibitem[Pet92]{Petkovsek92}
Marko Petkov\v{s}ek.
\newblock Hypergeometric solutions of linear recurrences with polynomial
  coefficients.
\newblock {\em J. Symbolic Comput.}, 14(2-3):243--264, 1992.

\bibitem[Phi15]{Phi15}
Patrice Philippon.
\newblock Groupes de {G}alois et nombres automatiques.
\newblock {\em J. Lond. Math. Soc. (2)}, 92(3):596--614, 2015.

\bibitem[Pou20]{poulet2020density}
Marina Poulet.
\newblock A density theorem for the difference galois groups of regular
  singular mahler equations.
\newblock {\em arXiv preprint arXiv:2012.14659}, 2020.

\bibitem[Roq11]{roques2011generalized}
Julien Roques.
\newblock Generalized basic hypergeometric equations.
\newblock {\em Inventiones mathematicae}, 184(3):499--528, 2011.

\bibitem[Roq18]{roques2018algebraic}
Julien Roques.
\newblock On the algebraic relations between mahler functions.
\newblock {\em Transactions of the American Mathematical Society},
  370(1):321--355, 2018.

\bibitem[RS06]{ramis2006q}
J-P Ramis and Jacques Sauloy.
\newblock The q-analogue of the wild fundamental group (i).
\newblock {\em arXiv preprint math/0611521}, 2006.

\bibitem[RS09]{RS2009}
Jean-Pierre Ramis and Jacques Sauloy.
\newblock The {$q$}-analogue of the wild fundamental group. {II}.
\newblock {\em Ast\'{e}risque}, 323:301--324, 2009.

\bibitem[RSZ13]{ramis2013local}
Jean-Pierre Ramis, Jacques Sauloy, and Changgui Zhang.
\newblock {\em Local analytic classification of q-difference equations}.
\newblock Soci{\'e}t{\'e} math{\'e}matique de France, 2013.

\bibitem[SS98]{springer1998linear}
Tonny~Albert Springer and TA~Springer.
\newblock {\em Linear algebraic groups}, volume~9.
\newblock Springer, 1998.

\bibitem[SU93]{singer1993galois}
Michael~F Singer and Felix Ulmer.
\newblock Galois groups of second and third order linear differential
  equations.
\newblock {\em Journal of Symbolic Computation}, 16(1):9--36, 1993.

\bibitem[VdPS97]{van2006galois}
Marius Van~der Put and Michael~F Singer.
\newblock {\em Galois theory of difference equations}.
\newblock Springer, 1997.

\bibitem[vH98]{vanHoeij1998}
Mark van Hoeij.
\newblock Rational solutions of linear difference equations.
\newblock In {\em Proceedings of the 1998 {I}nternational {S}ymposium on
  {S}ymbolic and {A}lgebraic {C}omputation ({R}ostock)}, pages 120--123. ACM,
  New York, 1998.

\end{thebibliography}
\end{document}